\documentclass[12pt, a4paper]{article}
\usepackage{amsmath}
\usepackage{comment}
\usepackage{amsfonts}
\usepackage{amssymb}
\usepackage{epsfig}
\usepackage{graphicx}
\usepackage{color}
\usepackage{enumerate}
\usepackage [latin1]{inputenc}
\usepackage[numbers, sort&compress]{natbib}
\usepackage{url}
\usepackage{amsthm,amsmath,amssymb}
\usepackage{mathrsfs}

\newtheorem{thm}{Theorem}[section]
\newtheorem{prop}[thm]{Proposition}

\newtheorem{lem}[thm]{Lemma}
\newtheorem{cor}[thm]{Corollary}
\newtheorem{ex}[thm]{Example}
\newtheorem{pb}[thm]{Problem}
\newtheorem{obs}[thm]{Observation}

\newtheorem{remark}[thm]{Remark}

\graphicspath{{figures/}}

\usepackage{url}

\usepackage{authblk}

\long\def\delete#1{}

\usepackage{xcolor}
\usepackage[normalem]{ulem}

\begin{document}
\openup 0.5\jot

\title{Components of domino tilings under flips in quadriculated tori}
\author[1,2]{Qianqian Liu\thanks{E-mail: \texttt{liuqq2016@lzu.edu.cn.}}}
\author[1]{Yaxian Zhang\thanks{E-mail: \texttt{yxzhang2016@lzu.edu.cn.}}}
\author[1]{Heping Zhang\footnote{The corresponding author. E-mail: \texttt{zhanghp@lzu.edu.cn.}}}

\affil[1]{\small School of Mathematics and Statistics, Lanzhou University, Lanzhou, Gansu 730000, China}
\affil[2]{\small College of Science, Inner Mongolia University of Technology, Hohhot, Inner Mongolia 010010, China}
\date{}

\maketitle

\setlength{\baselineskip}{20pt}
\noindent {\bf Abstract}:
In a region $R$ consisting of unit squares, a (domino) tiling is a collection of dominoes (the union of two adjacent squares) which pave fully the region. The flip graph of $R$ is defined on the set of all tilings of $R$ where two tilings are adjacent if we change one from the other by a flip (a $90^{\circ}$ rotation of a pair of side-by-side dominoes).
If $R$ is simply-connected, then its flip graph is connected.
By using homology and cohomology, Saldanha, Tomei, Casarin and Romualdo obtained a criterion to decide if two tilings are in
the same component of flip graph of quadriculated surface. By a graph-theoretic method, we obtain that the flip graph of a non-bipartite quadriculated torus consists of two isomorphic components. As an application, we obtain that the forcing numbers of all perfect matchings of each non-bipartite quadriculated torus form an integer-interval. For a bipartite quadriculated torus, the components of the flip graph is more complicated, and we use homology to obtain a general lower bound for the number of components of its flip graph.

\vspace{2mm} \noindent{\it Keywords}: Domino tiling; Flip; Perfect matching; Resonance graph; Quadriculated torus; Homology group.
\vspace{2mm}


\section{\normalsize Introduction}
In a quadriculated region $R$ consisting of unit squares with vertices in $\mathbb{Z}^2$, a \emph{domino} (or \emph{dimer}) is the union of two adjacent squares, and a (\emph{domino}) \emph{tiling} is a covering of $R$ by dominoes with disjoint interiors. In another language, $R$ can be transformed into a graph $D(R)$ called \emph{inner dual} of $R$: vertices of $D(R)$ are unit squares in $R$ and two vertices are adjacent if their corresponding squares have an edge in common. Then a domino of $R$ corresponds to an edge of $D(R)$ and a tiling corresponds to a perfect matching of $D(R)$.

A classical model is to determine the number of ways of an $n\times m$ chessboard ($nm$ even) can be covered by dimers, known as the \emph{dimer problem} \cite{Ka}, which
has important applications in both statistical mechanics and physical system. In other words, the dimer problem is to compute the number of perfect matchings in its inner dual. For rectangle lattices, the dimer problem was solved independently by Fisher \cite{Fisher}, Kasteleyn \cite{Ka}, Temperley and Fisher \cite{TF}. Further, Cohn et al. \cite{cohn} studied domino tilings of simply-connected regions of arbitrary shape on the plane by a variational method. Especially,
Thurston \cite{Th90} produced a necessary and sufficient condition for a simply-connected region in the plane to be tileable by dominoes.
There are also some results of dimer problem for 3-dimensional lattices \cite{RZ1,RZ2,lo,LC}.

Given a tiling of a region $R$, a \emph{flip} is a local move: lifting two parallel dominoes and then placing them back in a unique different position. Define the \emph{flip graph} $\mathcal{T}(R)$ on the set of all tilings of $R$ where two tilings are adjacent if we convert one from another by a flip. If $R$ is simply-connected, then $\mathcal{T}(R)$ is connected \cite{Th90,T1}. However, the result is not necessarily hold for other regions. Saldanha, Tomei, Casarin and Romualdo \cite{T1} gave three characterizations (combinational, homology and height section versions) for two tilings of a plane region in the same component of the flip graph. By homology and cohomology with local coefficients, they \cite{T1} provided a necessary and sufficient condition for two tilings in the same component of flip graphs of quadriculated surfaces.
By an elementary graph theory, we can obtain clearly the two components of flip graphs of non-bipartite quadriculated tori.

\begin{thm} \label{comnum}The flip graph of a non-bipartite quadriculated torus consists of two isomorphic components.
\end{thm}

Let $G$ be a graph with the vertex set $V(G)$ and edge set $E(G)$. A \emph{perfect matching} $M$ of $G$ is a subset of $E(G)$ such that each vertex is adjacent to exactly one edge in $M$. The \emph{forcing number} of  $M$ is the smallest cardinality of a subset of $M$ that is contained in no other perfect matching of $G$. This concept was first introduced by Klein and Randi\'{c} \cite{3} in chemical literatures, which has important applications in resonance theory. The \emph{forcing spectrum} of $G$ is the set of forcing numbers over all perfect matchings of $G$. As an application of Theorem \ref{comnum}, we obtain the following result.

\begin{thm} \label{qpp1} The forcing spectrum of a non-bipartite quadriculated torus is continuous (forming an integer-interval).
\end{thm}

For a bipartite quadriculated torus, we associate it to a CW-complex $\boldsymbol{T}$. By using the flux of a tiling i.e. homology class $[t-t_\oplus]$ in 1-dimensional homology group $H_1(\boldsymbol{T},Z)$, where $t_\oplus$ is a fixed base tiling and $t$ is any tiling of the quadriculated torus, we obtain that the flip graph of $T(n,m,r)$ (defined in Section 2) has at least $2\lfloor\frac{n}{2}\rfloor-1$ non-trivial components. In particular, the flip graph of usual $2n\times 2m$ quadriculated torus contains at least $2n+2m-3$ non-trivial components.

We organize the rest of paper as follows. In Section 2 we present some basic properties of quadriculated tori. In Section 3 we give the concept of resonance graphs, and prove Theorem \ref{comnum} using graph theory. In Section 4, we discuss the number of components of flip graphs for bipartite quadriculated tori. In Section 5, we prove Theorem \ref{qpp1} and give two examples whose forcing spectra are not continuous for bipartite quadriculated tori.

\section{\normalsize Properties of quadriculated tori}%
In this section, we give some basic properties of quadriculated tori. As early as 1991, Thomassen \cite{Tho} classified all quadriculated graphs on torus and Klein bottle.

\begin{thm}\rm{\cite{Tho}}\label{tho} Let $G$ be a connected 4-regular quadriculated graph on torus or Klein bottle. Then $G$ is one of the graphs $Q_{k,m,r}$, $Q_{k,m,a}$, $Q_{k,m,b}$, $Q_{k,m,c}$, $Q_{k,m,e}$, $Q_{k,m,f}$, $Q_{k,m,g}$ or $Q_{k,m,h}$, where $Q_{k,m,r}$ and $Q_{k,m,e}$ are graphs on torus and the others are graphs on Klein bottle.
\end{thm}

Further, Li \cite{classfy} proved that $Q_{k,m,r}$ and $Q_{k,m,e}$ can be reduced into one class of graphs.
\begin{thm}\rm{\cite{classfy}}\label{classfy} Any 4-regular quadriculated graph on torus is isomorphic to some $T(n,m,r)$.
\end{thm}

For $n\geq 1$ and $m\geq 2$, $T(n,m,r)$ represents a quadriculated torus obtained from an $n\times m$ chessboard by sticking the left and right sides together and then identifying the top and bottom sides with a torsion of $r$ squares (see Figure \ref{nota}(a)). Obviously, $T(n,m,0)=T(n,m,m)$ is the usual $n\times m$ torus. For convenience, we assume that  $1\leq r\leq m$.
\begin{figure}[h]
\centering
\includegraphics[height=4.5cm,width=16cm]{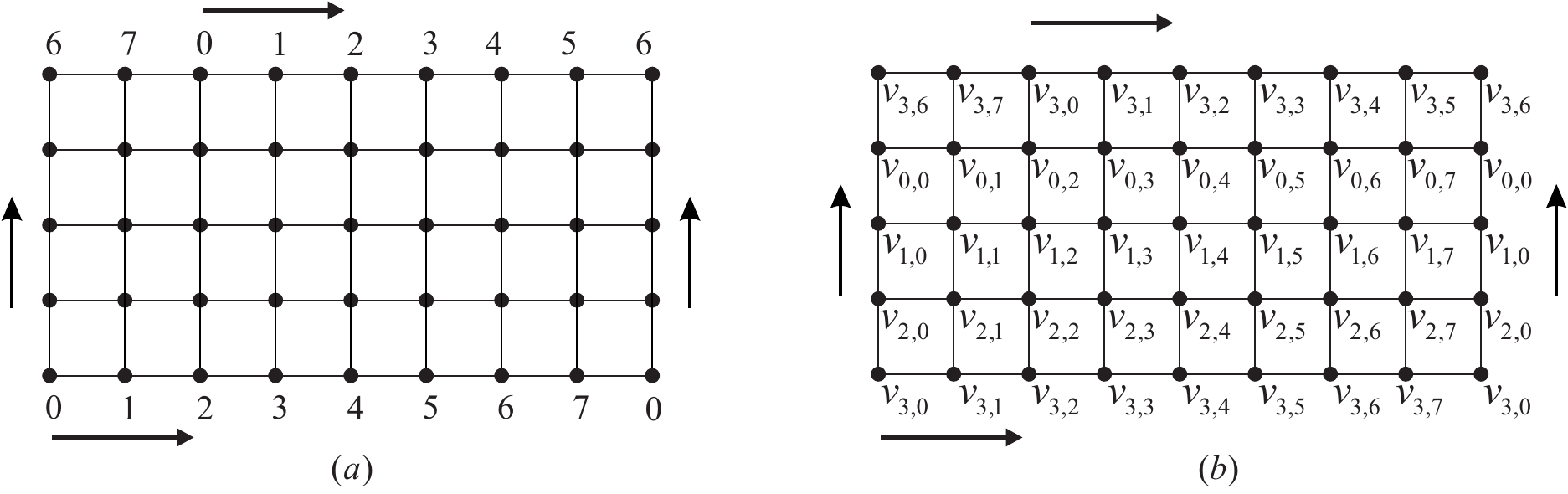}
\caption{\label{nota} Quadriculated torus $T(4,8,2)$ (left), and labels of its vertices (right).}
\end{figure}

For an integer $n\geq 1$, let $Z_n:=\{0,1,\dots,n-1\}$. According to positions of vertices in the chessboard, we label all vertices of $T(n,m,r)$ as $\{v_{i,j}| i\in Z_n, j \in Z_m\}$ (see Figure \ref{nota}(b)). Hence $v_{i,0}$ is adjacent to $v_{i,m-1}$ for $i\in Z_{n}$, and $v_{n-1,j}$ is adjacent to $v_{0,j+r}$ for $j\in Z_{m}$.

Define $\varphi(v_{i,j})=v_{i,j+1}$ where $i\in Z_n$ and $j\in Z_{m}$. Clearly, $\varphi$ is a translation of $T(n,m,r)$. By a simple argument, we obtain the following result.
\begin{lem}\label{auto} $\varphi$ is an automorphism on $T(n,m,r)$ which preserves the faces.
\end{lem}

A graph $G$ is \emph{bipartite} if $G$ has a 2-coloring, i.e. the vertices of $G$ are colored in black and white such that the ends of each edge receive different colors.
Noting that $T(n,m,r)$ is not necessarily bipartite, we can easily decide whether a quadriculated torus is bipartite or not.
\begin{thm}\label{bipartite} $T(n,m,r)$ is bipartite if and only if both $m$ and $n+r$ are even.
\end{thm}
\begin{proof}If both $m$ and $n+r$ are even, then the quadriculated cylinder obtained by identifying the left and right sides of the $n\times m$ chessboard is bipartite,
whose vertices can be colored with white and black such that two end vertices of all edges have different colors. Since $n$ and $r$ have the same parity, white (resp. black) vertices are glued with white (resp. black) vertices when identifying the top and bottom sides with a torsion of $r$ squares. Thus, two end vertices of all edges in $T(n,m,r)$ have different colors and $T(n,m,r)$ is bipartite.

Conversely, we will prove it by contradiction. If $m$ is odd, then the cycle induced by all vertices lying in a row contains $m$ vertices.
If $n+r$ is odd, then $v_{0,0}v_{1,0}\cdots v_{n-1,0}v_{n-1,m-1}\cdots \\v_{n-1,m-r}v_{0,0}$ is a cycle of length $n+r$. Hence $T(n,m,r)$ contains
an odd cycle (a cycle containing odd number of vertices), which contradicts that $T(n,m,r)$ is bipartite.
\end{proof}

Different from the $n\times m$ torus, all edges lying in a column of $T(n,m,r)$ do not necessarily form a cycle, whereas some columns may form a cycle. We next study such cycles in details.
For $j\in Z_m$, the path $v_{0,j}v_{1,j}\cdots v_{n-1,j}$ is called \emph{$j$-column} (or simply a \emph{column}),
and $v_{0,j}$ and $v_{n-1,j}$ are the \emph{initial} and \emph{terminal} of a $j$-column, respectively.
For $j_1$- and $j_2$-columns, if initial $v_{0,j_2}$ of $j_2$-column is
 adjacent to terminal $v_{n-1,j_1}$ of $j_1$-column, that is, $j_2\equiv j_1+r$ (mod $m$), then $j_2$-column is the \emph{successor} of $j_1$-column.
Let $j_0$-, $j_1$-, \dots, $j_{g-1}$-columns be pairwise different such that $j_{t+1}$-column is the successor of $j_t$-column for each $t\in Z_g$. Then these $g$ columns form a cycle, called an \emph{$I$-cycle}.

For $j\in Z_{m}$, let $j$ represent $j$-column of $T(n,m,r)$. Then there is a one-to-one correspondence between $m$ columns of $T(n,m,r)$ and $Z_m$. Let $p:Z_{m}\rightarrow Z_{m}$ be a permutation. If $p(j)=j+1$ for $j\in Z_{m}$, then $p$ is an \emph{$m$-cycle}, denoted by $p=(01\cdots m-1)$. Thus $I$-cycles of $T(n,m,r)$ correspond to cycle decomposition of $(01\cdots m-1)^r$. First we give some properties of $(01\cdots m-1)^r$. Note that $(r,m)$ represents the greatest common factor of $r$ and $m$.

\begin{lem}\label{cyclic} Let $p=(01\cdots m-1)$ be an $m$-cycle. For each integer $1\leq r \leq m$, the following results hold.

(\romannumeral1) \cite{cyclic} $p^r$ is the product of $(r,m)$ cycles each of length $\frac{m}{(r,m)}$.

(\romannumeral2) any consecutive $(r,m)$ numbers of $Z_m$ lie on different cycles of $p^r$.

(\romannumeral3) $j$ and $j+h$ belong to the same cycle of $p^r$ if and only if $(r,m)~|~h$.
\end{lem}
\begin{proof}Example 2.3 in \cite{cyclic} implies (\romannumeral1).

Next we prove (\romannumeral2). Suppose to the contrary that $j$ and $j+h$ lie on the same cycle where $j\in Z_m$ and $1\leq h\leq (r,m)-1$. That is to say, $j+h$ lies on the cycle which contains $j$. Then $j+h\equiv j+rt$ (mod $m$) for some $0\leq t \leq \frac{m}{(r,m)}-1$. Since $h< rt$, there exists a positive integer $u$ such that $h=rt-um$. So
\begin{equation}
1\leq \frac{m}{(r,m)}\leq h\cdot\frac{m}{(r,m)}\leq  ((r,m)-1)\cdot\frac{m}{(r,m)}= m-\frac{m}{(r,m)}. \label{rmr}
\end{equation}
On the other hand, $h\cdot\frac{m}{(r,m)}=(\frac{r}{(r,m)}\cdot t-\frac{m}{(r,m)}\cdot u)m \geq  m$, which is a contradiction to (\ref{rmr}).

Finally, we prove (\romannumeral3). If $j$ and $j+h$ belong to the same cycle of $p^r$, then $j+h\equiv j+rt$ (mod $m$) for some $0\leq t \leq \frac{m}{(r,m)}-1$. Thus, there exists an integer $u$ such that $h=rt+um$. Since $(r,m)~|~rt$ and $(r,m)~|~um$, we have $(r,m)~|~h$. Conversely, it suffices to prove that $j$ and $j+(r,m)$ lie on the same cycle of $p^r$ for $j\in Z_m$. By (\romannumeral2), consecutive $(r,m)$ numbers $j, j+1, \dots, j+(r,m)-1$ lie on different cycles. Since $p^r$ has $(r,m)$ cycles, $j+(r,m)$ lies on some cycle  containing $t$ where $t\in\{j,j+1, \dots, j+(r,m)-1\}$. By (\romannumeral2), $j+1,j+2, \dots, j+(r,m)$ lie on different cycles. Thus, $j+(r,m)$ and $j$ lie on the same cycle.
\end{proof}

By Lemma \ref{cyclic}, we obtain the following results immediately.

\begin{cor}\label{lem1}For $n\geq 1$, $m\geq 2$ and $1\leq r\leq m$, the following results hold.

(\romannumeral1) $T(n,m,r)$ has $(r,m)$ $I$-cycles and each $I$-cycle contains $\frac{m}{(r,m)}$ columns.

(\romannumeral2) any consecutive $(r,m)$ columns lie on different $I$-cycles.

(\romannumeral3) $j$ and $j+h$ belong to the same $I$-cycle if and only if $(r,m)~|~h$.
\end{cor}

The edges of $T(n,m,r)$ are intuitively classified into \emph{horizontal edges} and \emph{vertical edges}, as shown in Figure \ref{nota}. Obviously, all vertical edges form $(r,m)$ $I$-cycles. For $j\in Z_m$, let $E_j=\{e_{i,j}=v_{i,j}v_{i,j+1}|i\in Z_n\}$ be the set of all horizontal edges whose end vertices lie on $j$-column and $(j+1)$-column, respectively. For $j\in Z_{(r,m)}$, let $C_j$ be the $I$-cycle of $T(n,m,r)$ containing the $j$-column. Then $C_j$ consists of $j$-, $(j+r)$-, \dots, $(j+r(\frac{m}{(r,m)}-1))$-columns, and $E_j\cup E_{j+r}\cup \dots \cup E_{j+(\frac{m}{(r,m)}-1)r}$ is the set of all horizontal edges, which each has one end in $C_j$ and the other in $C_{j+1}$, where subscripts take modulo $m$.
The union of $E(C_j)\cup E(C_{j+1})$ and $E_j\cup E_{j+r}\cup \cdots \cup E_{j+(\frac{m}{(r,m)}-1)r}$ forms a subgraph, denoted by $R_{j,j+1}$, which is the Cartesian product of a cycle of length $\frac{nm}{(r,m)}$ with $K_2$.

For $j\in Z_{(r,m)}$, from the vertex $v_{0,j}$, we in turn relabel the vertices of $C_j$ as $1^j$, $2^j$, \dots, $(\frac{mn}{(r,m)})^j$ where $1^j=v_{0,j}$, $2^j=v_{1,j}$, and so on. Then $n$ vertices of $(j+rt)$-column are labeled as $(tn+1)^j$, $(tn+2)^j$, \dots, $((t+1)n)^j$ from top to bottom (see Figure \ref{labels}), where $0\leq t\leq \frac{m}{(r,m)}-1$.
\begin{figure}[h]
\centering
\includegraphics[height=3.2cm,width=11cm]{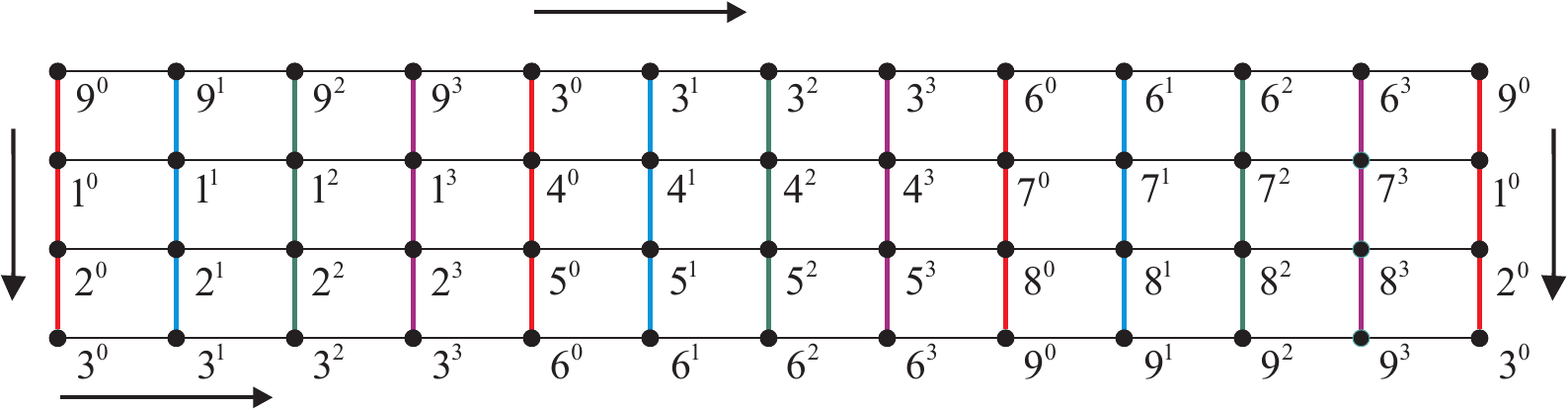}
\caption{\label{labels}New labels of $T(3,12,4)$ according to $I$-cycles}
\end{figure}

Further, we can obtain the relationship between labels of end vertices of horizontal edges.

\begin{lem}\label{prop} Let $x^jy^{j+1}$ be a horizontal edge of $T(n,m,r)$ where $1\leq x,y\leq \frac{mn}{(r,m)}$ and $j\in Z_{(r,m)}$. Then

(\romannumeral1) For $0\leq j\leq (r,m)-2$, we have $x=y$.

(\romannumeral2) For $j=(r,m)-1$, we have
\begin{equation}
 y=
 \begin{cases}
 x+nk, & \quad {if\  1\leq x\leq (\frac{m}{(r,m)}-k)n};\\
 x-(\frac{m}{(r,m)}-k)n,&\quad {otherwise}, \label{endvertex}
 \end{cases}
 \end{equation}where $0\leq k\leq \frac{m}{(r,m)}-1$ is an integer such that $(r,m)\equiv rk$ $($mod $m)$.
\end{lem}
\begin{proof}Since $x^j\in V(C_j)$ and $y^{j+1}\in V(C_{j+1})$, $x^jy^{j+1}$ is a horizontal edge of $R_{j,j+1}$. Since all horizontal edges of $R_{j,j+1}$ consist of $E_j\cup E_{j+r}\cup \cdots \cup E_{j+r(\frac{m}{(r,m)}-1)}$, there exists $t\in Z_{\frac{m}{(r,m)}}$ such that $x^jy^{j+1}\in E_{j+rt}$. Noting that an edge of $E_{j+rt}$ has one end vertex on $(j+rt)$-column and another on $(j+1+rt)$-column, we will obtain the relationship between $x$ and $y$ from the labels of its two end vertices.

For $0\leq j\leq (r,m)-2$, since all vertices of $(j+rt)$-column are labeled as $(nt+1)^j$, $(nt+2)^j$, \dots, $(n(t+1))^j$ and these vertices of $(j+1+rt)$-column are labeled as $(nt+1)^{j+1}$, $(nt+2)^{j+1}$, \dots, $(n(t+1))^{j+1}$ from top to bottom, we obtain that $y=x$.

But for $j=(r,m)-1$, we have $j+1=(r,m)$. By Corollary \ref{lem1}(\romannumeral3), $(r,m)$-column lies on $C_{0}$. Since $C_0$ consists of 0-, $r$-, $2r$-, $\cdots$, $(\frac{m}{(r,m)}-1)r$-columns, we have $(r,m)\equiv rk$ (mod $m$) for some integer $k$ with $0\leq k\leq \frac{m}{(r,m)}-1$.
Thus, $x^j$ lies on $(j+rt)$-column and $y^{j+1}$ lies on $r(k+t)$-column. Since all vertices of $(j+rt)$-column are labeled as $(nt+1)^j$, $(nt+2)^j$, \dots, $(n(t+1))^j$ and these vertices of $((k+t)r)$-column are labeled as $((k+t)n+1)^{0}$, $((k+t)n+2)^{0}$, \dots, $(n(k+t+1))^{0}$ from top to bottom, we obtain that $y=x+kn$ (mod $\frac{mn}{(r,m)}$).
That is to say, for $1\leq x\leq (\frac{m}{(r,m)}-k)n$, we have $y=x+kn$. Otherwise, we have $\frac{mn}{(r,m)}+1\leq x+kn \leq \frac{mn}{(r,m)}+kn\leq \frac{2mn}{(r,m)}-n$. Thus,
we obtain that $y=x+kn-\frac{mn}{(r,m)}=x-(\frac{m}{(r,m)}-k)n$.
\end{proof}

From the above discussion we can give another representation for $T(n,m,r)$.
Given an $\frac{mn}{(r,m)}\times (r,m)$ chessboard, by Lemma \ref{prop}(\romannumeral1), we label the vertices on the $(i+1)$-column from top to bottom as $\{(\frac{mn}{(r,m)})^i$, $1^i$, $2^i$, $\dots$, $(\frac{mn}{(r,m)})^i\}$ for $i\in Z_{(r,m)}$, and label the vertices of the last column as $(nk)^0$, $(nk+1)^0$, $(nk+2)^0$, $\dots$, $(\frac{mn}{(r,m)})^0$, $1^0$, $2^0$, \dots, $n^0$, \dots, $(n(k-1)+1)^0$, $(n(k-1)+2)^0$, \dots, $(nk)^0$ from top to bottom.
In order to form the same pattern as $T(n,m,r)$, we stick the top and bottom sides together, and the left and right sides by a torsion of $(\frac{m}{(r,m)}-k)n$ squares.
By definition, this is the quadriculated torus $T((r,m), \frac{mn}{(r,m)},(\frac{m}{(r,m)}-k)n)$, denoted by $T^*(n,m,r)$. That is, $T^*(n,m,r)=T((r,m), \frac{mn}{(r,m)},(\frac{m}{(r,m)}-k)n)$. In the above process, all horizontal and vertical edges of $T(n,m,r)$ are still horizontal and vertical in $T^*(n,m,r)$, and all faces of $T(n,m,r)$ are mapped into faces of $T^*(n,m,r)$. Thus, this is an equivalence between such two quadriculated tori.
For example, we give $T^*(3,12,4)=T(4,9,6)$ (see Figure \ref{obse}) where $k=1$, and $T^*(3,12,10)=T(2,18,3)$ where $k=5$. In fact, $k$ could be any integer from 0 to $\frac{m}{(r,m)}-1$.
\begin{figure}[h]
\centering
\includegraphics[height=7cm,width=4.6cm]{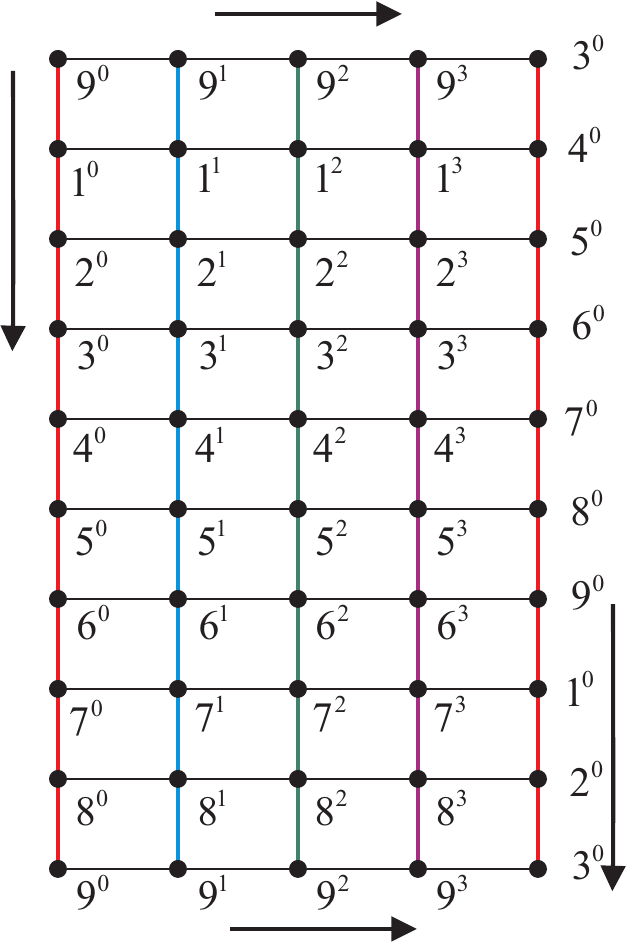}
\caption{\label{obse}Quadriculated torus $T(4,9,6)$, a dual representation of $T(3,12,4)$.}
\end{figure}

The next proposition shows that $T^*(n,m,r)$ is a dual.
\begin{prop}\label{drawing} For $n,m\geq 1$ and $1\leq r\leq m$, $T^{**}(n,m,r)=T(n,m,r)$.
\end{prop}
\begin{proof}Let $n'=(r,m)$, $m'=\frac{mn}{(r,m)}$, $r'=n(\frac{m}{(r,m)}-k)$. Then $T^*(n,m,r)=T(n', m',r')$, where $0\leq k\leq \frac{m}{(r,m)}-1$ is an integer such that $(r,m)\equiv rk$ $($mod $m)$. Note that $$T^{**}(n,m,r)=T^*(n', m',r')=T((r',m'), \frac{m'n'}{(r',m')},n'(\frac{m'}{(r',m')}-k')),$$ where $0\leq k'\leq \frac{m'}{(r',m')}-1$ is an integer such that $(r',m')\equiv r'k'$ $($mod $m')$. It suffices to prove that $(r',m')=n$, $\frac{m'n'}{(r',m')}=m$ and $n'(\frac{m'}{(r',m')}-k')=r$.

Since $$(r',m')=(n(\frac{m}{(r,m)}-k),\frac{mn}{(r,m)})=n(\frac{m}{(r,m)}-k,\frac{m}{(r,m)})=n(\frac{m}{(r,m)},-k),$$ we only need to prove that $(\frac{m}{(r,m)},-k)=1$. Since $rk\equiv (r,m)$ $($mod $m)$, we have $\frac{r}{(r,m)}k\equiv 1$ $($mod $\frac{m}{(r,m)})$. That is, there exists an integer $p$ such that $\frac{r}{(r,m)}k-1=\frac{m}{(r,m)}p$. Hence, we obtain that $\frac{-r}{(r,m)}(-k)-p\frac{m}{(r,m)} =1$. Thus $-k$ and $\frac{m}{(r,m)}$ are mutually prime, and $(\frac{m}{(r,m)},-k)=1$. Therefore, $(r',m')=n$ and $\frac{m'n'}{(r',m')}=\frac{\frac{mn}{(r,m)}(r,m)}{n}=m$.

Noting that $$n'(\frac{m'}{(r',m')}-k')=(r,m)(\frac{\frac{mn}{(r,m)}}{n}-k')=(r,m)(\frac{m}{(r,m)}-k')=m-(r,m)k',$$ we need to prove that $k'=\frac{m-r}{(r,m)}$. By definition, we have $r'k'\equiv (r',m')$ (mod $m'$). That is, $k'$ is a solution of $n(\frac{m}{(r,m)}-k)x \equiv n$ (mod $\frac{mn}{(r,m)}$). Dividing by $n$, we obtain that $(\frac{m}{(r,m)}-k)x \equiv 1 $ (mod $\frac{m}{(r,m)}$). By the above proof, we have $(\frac{m}{(r,m)}-k,\frac{m}{(r,m)})=1$. Thus $(\frac{m}{(r,m)}-k)x\equiv 1$ (mod $\frac{m}{(r,m)}$) has a unique solution (see Theorem 57 in \cite{hardy}). Finally, we will check that $k'=\frac{m-r}{(r,m)}$ is a solution of above equation. Recall that $\frac{rk}{(r,m)}\equiv 1 ~(\text{mod } \frac{m}{(r,m)}),$ we have
$$(\frac{m}{(r,m)}-k)\frac{m-r}{(r,m)} \equiv (-k)\frac{m-r}{(r,m)}  \equiv\frac{rk}{(r,m)}\equiv 1 ~(\text{mod } \frac{m}{(r,m)}).$$ Thus $k'=\frac{m-r}{(r,m)}$ is the unique solution of $(\frac{m}{(r,m)}-k)x \equiv 1$ (mod $\frac{m}{(r,m)}$), and we finish the proof.
\end{proof}

\section{\normalsize Non-bipartite quadriculated tori}
In this section, we give the concept of resonance graphs and prove Theorem \ref{comnum} by using the theory of resonance graphs.

\subsection{\small Resonance graphs}

For a perfect matching $M$ of a graph $G$, a cycle $C$ of $G$ is $M$-\emph{alternating} if its edges appear alternately in $M$ and $E(C)\setminus M$. For an $M$-alternating cycle $C$ of $G$, the \emph{symmetric difference} $M\oplus E(C):=(M-E(C))\cup (E(C)-M)$ is another perfect matching of $G$ and such an operation is called a \emph{rotation} of $M$ (along $C$).

Let $G$ be a graph 2-cell embedded on a closed surface, i.e.  each face is homeomorphic to an open disc in $\mathbb{R}^2$. For convenience, sometimes we do not distinguish a face with its boundary, and a cycle with its edge set.
The \emph{resonance graph} of $G$, denoted by $R_t(G)$, is defined on the set of all perfect matchings of $G$ where two perfect matchings $M_1$ and $M_2$ are adjacent if and only if $M_1\oplus M_2$ forms the boundary of a face, that is, $M_2$ is obtained from a rotation of $M_1$ along the boundary of a face, simply a flip of $M_1$.
The resonance graph was introduced independently many times under different names by Gr\"{u}mdler \cite{G}, Zhang et al. \cite{Z1}, Randi\'{c} \cite{ran} and Fournier \cite{fou}. The study of resonance graphs mainly concentrates on some special plane graphs such as hexagonal systems \cite{G,Z1,ran}, polyominoes \cite{Z4,Z6} and plane bipartite graphs \cite{fou,Z3}. Readers may refer to a survey \cite{Zh06}.

Note that the flip graph of a quadriculated region corresponds to the resonance graph of its inner dual. Since $T(n,m,r)$ is a 4-regular quadriculated graph on torus, so is its (inner) dual. By Theorem \ref{classfy}, $D(T(n,m,r))$ is also a quadriculated torus. Analyzing carefully, we find that inner dual of $T(n,m,r)$ is itself. Hence, we have the following observation.

\begin{obs}\label{obse2}The flip graph of $T(n,m,r)$ is the same as its resonance graph.
\end{obs}

\subsection{\small Preliminaries}
Let $M$ be a perfect matching of $T(n,m,r)$ consisting of vertical edges. For $j,k\in Z_{m}$ and $j\neq k$, if both $v_{0,j}v_{1,j}$ and $v_{0,k}v_{1,k}$ belong to $M$ or neither $v_{0,j}v_{1,j}$ nor $v_{0,k}v_{1,k}$ belongs to $M$, then we say that $j$- and $k$-columns \emph{have the same form on $M$}. For $i\in Z_{(r,m)}$, let $C_i$ and $C_{i+1}$ be two consecutive $M$-alternating $I$-cycles. Then $M$ \emph{has the same form} on $C_i$ and $C_{i+1}$ if and only if there exists $k$-column on $C_i$ and $(k+1)$-column on $C_{i+1}$ for some $k\in Z_m$ such that $M$ have the same form on such two columns.
For a set of $I$-cycles $\mathcal{C}$, if any two consecutive ones of $\mathcal{C}$ have the same form on $M$, then we say that $\mathcal{C}$ \emph{has the same form on $M$}.

By a simple argument on $R_{j,j+1}$, we obtain the following lemma.
\begin{lem}\label{le61}For $j\in Z_{(2r,2m)}$, let $M$ be a perfect matching of $R_{j,j+1}\subseteq T(2n+1,2m,2r)$.

(\romannumeral1) If $M$ contains exactly one horizontal edge, then, by $\frac{m(2n+1)-(r,m)}{2(r,m)}$ flips from $M$, we can obtain a perfect matching of $R_{j,j+1}$ consisting of all horizontal edges.

(\romannumeral2) If $M$ consists of vertical edges and $C_{j}$ and $C_{j+1}$ have the same form on $M$, then, by $\frac{m(2n+1)}{2(r,m)}$ flips from $M$, we can obtain a perfect matching of $R_{j,j+1}$ consisting of all horizontal edges.

(\romannumeral3) Moreover, by $\frac{m(2n+1)}{2(r,m)}$ flips on the resulting perfect matching in (\romannumeral2), we can obtain a perfect matching $M'=(E(C_j)\cup E(C_{j+1}))\oplus M$.
\end{lem}

For a vertex subset $X\subseteq V(G)$, the set of edges with exactly one end vertex in $X$ is written $\nabla(X)$. For a subgraph $G'$ of $G$, we write $\nabla(G')$ instead of $\nabla(V(G'))$.
Let $M$ be a perfect matching of $T(2n+1,2m,2r)$. For $i\in Z_{2(r,m)}$, a horizontal edge of $M\cap \nabla(C_i)$ lies in either $R_{i-1,i}$ or $R_{i,i+1}$. For $k\in\{1,2\}$, let $h_k$ be a horizontal edge of $M\cap \nabla(C_i)$ such that $x_k$ is the end vertex of $h_k$ on $C_i$.
We call $h_1$ and $h_2$ \emph{consecutive} if there exists a path $P$ on $C_i$ connecting $x_1$ and $x_2$ such that all intermediate vertices are incident to edges of $M\cap E(P)$.

By a similar proof as Lemma 4.2 in \cite{LW}, we can obtain the following result. Here, we give a simpler proof.
\begin{lem}\label{l3} Suppose that $M$ is a perfect matching of $T(2n+1,2m,2r)$ and $C$ is an $I$-cycle. If there exist two consecutive horizontal edges of $M\cap \nabla(C)$ on the same side, then we can obtain a perfect matching $M'$ of $T(2n+1,2m,2r)$ by some flips from $M$ such that the number of horizontal edges in $M'$ is two less than that in $M$.
\end{lem}
\begin{proof} For an $I$-cycle $C$, let $h_1$ and $h_2$ be two consecutive horizontal edges of $M\cap \nabla(C)$ on the same side, where $x_k$ is the end vertex of $h_k$ on $C$. Since $h_1$ and $h_2$ are consecutive, let $P$ be a path of $C$ between $x_1$ and $x_2$ such that all intermediate vertices are incident to edges of $M\cap E(P)$. Choose such a shortest path $P$ among all $I$-cycles $C$.

Without loss of generality, we assume that $P=v_{0,j}v_{1,j}\cdots v_{l,j}$ for some $j\in Z_{2m}$. Then $\{v_{0,j}v_{0,j+1}, v_{1,j}v_{2,j}, v_{3,j}v_{4,j}, \dots, v_{l-2,j}v_{l-1,j},v_{l,j}v_{l,j+1}\}\subseteq M$ and $l$ is odd. If $l=1$, then by a flip along $v_{0,j}v_{1,j}v_{1,j+1}v_{0,j+1}v_{0,j}$, we obtain a perfect matching $M'$ which has two horizontal edges less than $M$. If $l\geq 3$, then we claim that $|M\cap \nabla(Q)|=0$, where $Q=v_{1,j+1}v_{2,j+1}\cdots v_{l-1,j+1}$ is a path.
Suppose to the contrary that $|M\cap \nabla(Q)|\neq 0$.
Since $Q$ contains an even number of vertices, $|M\cap \nabla(Q)|\geq 2$ is even. Let $v_{s,j+1}v_{s,j+2}$ and $v_{t,j+1}v_{t,j+2}$ be two consecutive horizontal edges of $M\cap \nabla(Q)$, where $1\leq s<t\leq l-1$. Then $Q'=v_{s,j+1}v_{s+1,j+1}\cdots v_{t,j+1}$ is a sub-path of $Q$ that is shorter than $P$, which is a contradiction.

Let $S_i=v_{2i+1,j}v_{2i+2,j}v_{2i+2,j+1}v_{2i+1,j+1}v_{2i+1,j}$ and $T_k=v_{2k,j}v_{2k+1,j}v_{2k+1,j+1}v_{2k,j+1}v_{2k,j}$ be a face of $T(2n+1,2m,2r)$, respectively, for $i\in Z_{\frac{l-1}{2}}$ and $k\in Z_{\frac{l+1}{2}}$. Then $M'=M\oplus S_0\oplus S_1 \oplus \cdots \oplus S_{\frac{l-3}{2}}\oplus T_0\oplus T_1 \oplus \cdots \oplus T_{\frac{l-1}{2}}$ is a perfect matching of $T(2n+1,2m,2r)$ such that the number of horizontal edges of $M'$ is two less than that of $M$.
\end{proof}

By Theorem \ref{bipartite}, the non-bipartite quadriculated tori have three cases: $T(2n+1,2m,2r)$, $T(2n,2m,2r-1)$ and $T(2m,2n+1,r)$. Next we mainly prove Theorem \ref{comnum} for $T(2n+1,2m,2r)$, and use Proposition \ref{drawing} to reduce the other two cases into this case.

\subsection{\small Proof for the quadriculated torus $T(2n+1,2m,2r)$}
Let $M_1=E_0\cup E_2\cup \cdots \cup E_{2m-2} \text{ and } M_2=E_1\cup E_3\cup \cdots\cup E_{2m-1}$ be two special perfect matchings of $T(2n+1,2m,2r)$. Then we have the following lemma.
\begin{lem}\label{lem11} For $n\geq1$, $m\geq 2$ and $1\leq r\leq m$, $M_1$ and $M_2$ lie in different components of $R_t(T(2n+1,2m,2r))$.
\end{lem}
\begin{proof}Suppose to the contrary that $M_1$ and $M_2$ lie in the same component of $R_t(T(2n+1,2m,2r))$. Then there is a path in $R_t(T(2n+1,2m,2r))$ connecting $M_1$ and $M_2$. That is, $M_2=M_1\oplus S_1\oplus\cdots \oplus S_k$ where $S_i$ is a square of $T(2n+1,2m,2r)$ for $i=1,2,\dots,k$.
For a square $f$, we denote by $\delta(f)$ the number of times that $f$ appears in the sequence $(S_1,S_2,\dots,S_k)$.
\begin{figure}[h]
\centering
\includegraphics[height=3.8cm,width=7cm]{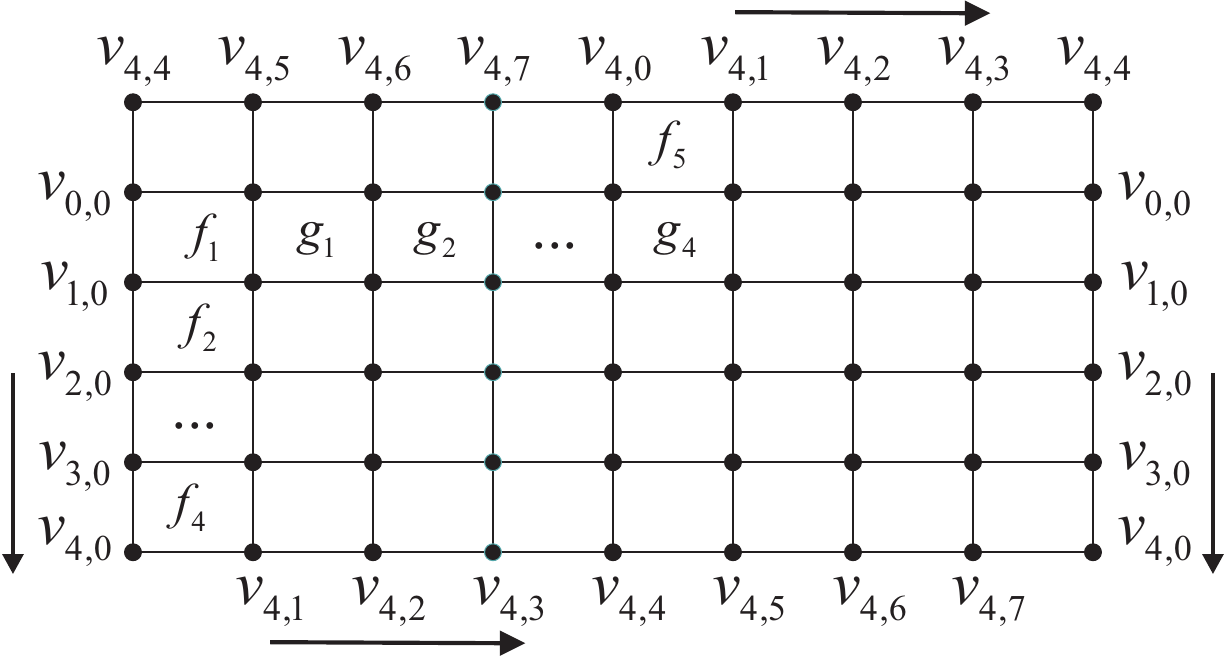}
\caption{\label{atleasttwo}Some squares of $T(5,8,4)$.}
\end{figure}

For $1\leq i\leq 2n$, let $f_{i}$ and $f_{i+1}$ be the two faces of $T(2n+1,2m,2r)$ having edge $e_{i,0}$ in common shown as Figure \ref{atleasttwo}.
Since $e_{i,0}\in M_1$ but $e_{i,0}\notin M_2$, $e_{i,0}$ appears odd times in the sequence $(S_1,S_2,\dots,S_k)$. Thus
\begin{eqnarray}
\delta(f_{i})+\delta(f_{i+1})\equiv 1\text{ (mod 2) } \text{for } 1\leq i\leq 2n.\label{eq1}
\end{eqnarray}

Let $g_i$ be a face of $T(2n+1,2m,2r)$ bounded by $v_{0,i}v_{1,i}v_{1,i+1}v_{0,i+1}v_{0,i}$ for $1\leq i\leq 2r$. Then $f_{2n+1}$ and $g_{2r}$ have
exactly one edge $e_{0,2r}$ in common, and $g_{i-1}$ and $g_i$ have the common edge $v_{0,i}v_{1,i}$ for $1\leq i\leq 2r$ where $g_0:=f_1$.
Since $e_{0,2r}\in M_1$ but $e_{0,2r}\notin M_2$, $e_{0,2r}$ appears odd times in the sequence $(S_1,S_2,\dots,S_k)$. So we have
\begin{eqnarray}
\delta(f_{2n+1})+\delta(g_{2r})\equiv 1 \text{ (mod 2)}.\label{eq2}
\end{eqnarray}

Since $v_{0,i}v_{1,i}$ lies in neither $M_1$ nor $M_2$ for $1\leq i\leq 2r$, we obtain that
\begin{eqnarray}
\delta(g_{i-1})+\delta(g_i)\equiv 0 \text{ (mod 2)} \text{ for } 1\leq i\leq 2r.\label{eq3}
\end{eqnarray}

Now, we add up $2n+2r+1$ equalities in $(\ref{eq1})$-$(\ref{eq3})$ together. The resulting right side is odd since $2n+1$ equalities equal to 1 and the others equal to 0, but the left side is even since each $\delta (f)$ appears twice for $f\in \{f_1,f_2\dots,f_{2n+1},g_1,\dots,g_{2r}\}$, which is a contradiction.
\end{proof}

Lemma \ref{lem11} states that $R_t(T(2n+1,2m,2r))$ has at least two components. In the sequel, we shall prove that $R_t(T(2n+1,2m,2r))$ has at most two components. That is to say, any perfect matching $M$ of $T(2n+1,2m,2r)$ can be converted into the perfect matching $M_1$ or $M_2$ by some flips. First of all, we consider the case that $M$ consists of vertical edges.

\begin{lem}\label{lem61}For $n\geq1$, $m\geq 2$ and $1\leq r\leq m$, if $M$ is a perfect matching of $T(2n+1,2m,2r)$ consisting of vertical edges, then we can obtain the perfect matching $M_1$ or $M_2$ by a series of flips from $M$.
\end{lem}
\begin{proof}
Since $M$ consists of vertical edges, each $I$-cycle is even and $M$-alternating. Hence, we obtain that $r\leq m-1$. By Corollary \ref{lem1}(\romannumeral1), $T(2n+1,2m,2r)$ has $t:=2(r,m)$ $I$-cycles. By  Corollary \ref{lem1}(\romannumeral3), $C_0$ consists of $0$-, $t$-, \dots, $(\frac{m}{(r,m)}-1)t$-columns.
Because $C_0$ is $M$-alternating and each column contains $2n+1$ vertices, it has half of columns being the same form on $M$ and the other half being the other form on $M$. Thus, there exist $x$- and $(x+t)$-columns having different form on $M$ where $x\in Z_{2m}$.  Without loss of generality, we assume that $x=0$ and $\{v_{0,0}v_{1,0},v_{0,t}v_{2n,2m-2r+t}\}\subseteq M$.

Let
$\alpha=|M\cap \{v_{0,1}v_{2n,2m-2r+1},v_{0,2}v_{2n,2m-2r+2},\dots,v_{0,t-1}v_{2n,2m-2r+t-1}\}|$.
Then $0\leq \alpha\leq t-1$. We will prove the lemma by induction on $\alpha$. For $\alpha=0$, all $t$ $I$-cycles have the same form on $M$. Since $M\cap E(R_{2l,2l+1})$ is a perfect matching of $R_{2l,2l+1}$ for $0\leq l\leq \frac{t}{2}-1$, by repeatedly applying Lemma \ref{le61}(\romannumeral2) $\frac{t}{2}$ times, we can obtain the perfect matching $M_1$.
\begin{figure}[h]
\centering
\includegraphics[height=2.5cm,width=16cm]{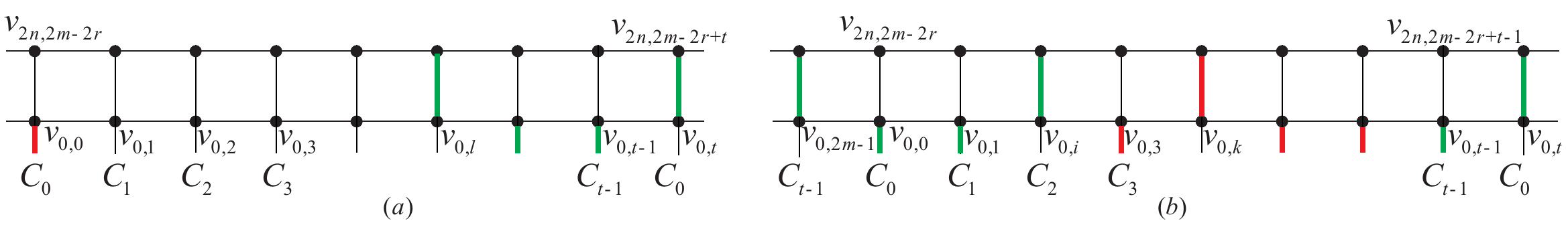}
\caption{\label{icycle}Description of the proof of Lemma \ref{lem61}.}
\end{figure}

Next we assume that $\alpha\geq 1$. If there exist two positive integers $l$ and $s$ with $1\leq l<l+s\leq t$ such that
$\{v_{0,l}v_{2n,2m-2r+l},v_{0,l+s}v_{2n,2m-2r+l+s}\}\cup \{v_{0,l+1}v_{1,l+1},\dots,v_{0,l+s-1}v_{1,l+s-1}\}\subseteq M$ shown as Figure \ref{icycle}(a), and $s-1$ is an even number, then $\{C_{l+1}, C_{l+2}, \dots, C_{l+s-1}\}$ have the same form on $M$. For $R_{l+1,l+2},R_{l+3,l+4},\dots,R_{l+s-2,l+s-1}$, we separately perform $\frac{m(2n+1)}{(r,m)}$ flips from its perfect matching by Lemma \ref{le61}(\romannumeral3) and last result in a perfect matching $M'$ consisting of vertical edges. Now $\{C_l, C_{l+1},  \dots,C_{l+s}\}$ have the same form on $M'$. By performing $\frac{m(2n+1)}{(r,m)}$ flips separately on $R_{l,l+1},R_{l+2,l+3},\cdots,R_{l+s-1,l+s}$ by Lemma \ref{le61}(\romannumeral3), we result in a perfect matching $M''$ which consists of vertical edges and $\alpha''=\alpha-1$ or $\alpha-2$.
By the induction hypothesis, we can obtain the perfect matching $M_1$ or $M_2$ from $M''$ by some flips. Thus, $M_1$ or $M_2$ can be obtained from $M$ by a series of flips.

Now we assume that $s-1$ is an odd number if there exist above $l$ and $s$. Thus $v_{0,t-1}v_{1,t-1}\in M$.
Since $\alpha\geq1$, there exist the maximum index $k$ with $1\leq k\leq t-2$ such that $v_{0,k}v_{2n,2m-2r+k}\in M$ shown in Figure \ref{icycle}(b). Then all edges of $\{v_{0,k+1}v_{1,k+1},\dots,v_{0,t-2}v_{1,t-2}\}$ are contained in $M$. Take $l=k$ and $l+s=t$. Since $s-1$ is an odd number, $s$ is an even number and so is $k$.
Let $i$ be the minimum index such that $v_{0,i}v_{2n,2m-2r+i}\in M$ where $1\leq i\leq t-2$. Then $\{v_{0,1}v_{1,1},\dots, v_{0,i-1}v_{1,i-1}\}\subseteq M$ and $i$ is an even number. Since $\{v_{0,t-1}v_{1,t-1},v_{0,t}v_{2n,2m-2r+t}\}\subseteq M$, $C_{t-1}$ and $C_0$ have different form on $M$. Combining the condition that  $v_{0,0}v_{1,0}\in M$, we have $v_{0,2m-1}v_{2n,2m-2r-1}\in M$.
 Now $\{C_0,C_1,\dots, C_{i-1}\}$ have the same form on $M$. We
perform $\frac{m(2n+1)}{(r,m)}$ flips on each $R_{0,1}, R_{2,3},\dots, R_{i-2,i-1}$ by Lemma \ref{le61}(\romannumeral3) and result in a perfect matching $M'$ consisting of vertical edges. Then $\{C_{t-1},C_0,\dots, C_{i}\}$ have the same form on $M'$. By repeatedly applying Lemma \ref{le61}(\romannumeral3) on each $R_{t-1,0}, R_{1,2},\dots,R_{i-1,i}$, we can obtain a perfect matching $M''$ by some flips from $M'$ such that $\alpha''=\alpha$ but $v_{0,t-1}v_{2n,2m-2r+t-1}\in M''$. Taking $l=t-1$ and $s=1$, we obtain that $s-1$ is an even number. By previous arguments, we have done.
\end{proof}

Now we prove that any perfect matching of $T(2n+1,2m,2r)$ can be converted into the perfect matching $M_1$ or $M_2$ by a series of flips.

\begin{lem}\label{l1}Let $M$ be a perfect matching of $T(2n+1,2m,2r)$ where $n\geq1$, $m\geq 2$ and $1\leq r\leq m$. Then, by a series of flips from $M$, we can obtain the perfect matching $M_1$ or $M_2$.
\end{lem}
\begin{proof}Let $l$ represent the number of horizontal edges in $M$. Then $l\geq 0$. We will prove the lemma by induction on $l$. For $l=0$, $M$ consists of vertical edges. By Lemma \ref{lem61}, we have done. For $l=1$, we obtain that all $I$-cycles contain an odd number of vertices by  Corollary \ref{lem1}(\romannumeral1). Moreover, $T(2n+1,2m,2r)$ has exactly two $I$-cycles. We assume that the unique horizontal edge of $M$ belongs to $R_{0,1}$ (resp. $R_{1,0}$). Then $M$ is uniquely determined, by Lemma \ref{le61}(\romannumeral1), we can obtain the perfect matching $M_1$ (resp. $M_2$) by  $\frac{m(2n+1)-(r,m)}{2(r,m)}$ flips.
Now we assume that $l\geq 2$. We divide into the following two cases according to the parity of the number of vertices in an $I$-cycle.

\textbf{Case 1.} An $I$-cycle is odd.

For an $I$-cycle $C$, $|M\cap \nabla(C)|$ is odd. We assume that $M\cap \nabla(C)$ contains exactly one edge for each $I$-cycle $C$. Then $T(2n+1,2m,2r)$ has exactly $(r,m)$ horizontal edges in $M$. Without loss of generality, we assume that $e_i$ is the unique horizontal edge of $R_{2i,2i+1}$ for $i\in Z_{(r,m)}$. By Lemma \ref{le61}(\romannumeral1), we can obtain a perfect matching $M'$ from $M$ by $\frac{m(2n+1)-(r,m)}{2(r,m)}$ flips which consists of all horizontal edges of $R_{2i,2i+1}$. Hence, by $\frac{m(2n+1)-(r,m)}{2}$ flips from $M$, we can obtain the perfect matching $M_1$.
Now we assume that there exists an $I$-cycle $C_i$ of $T(2n+1,2m,2r)$ such that $M\cap \nabla(C_i)$ contains at least three edges where $i\in Z_{(2r,2m)}$. Since an edge of $M\cap \nabla(C_i)$ belongs to either $R_{i-1,i}$ or $R_{i,i+1}$ and $|M\cap \nabla(C_i)|$ is an odd number, $T(2n+1,2m,2r)$ contains two consecutive horizontal edges on the same side.

\textbf{Case 2.} An $I$-cycle is even.

We will prove that there exists an $I$-cycle $C_i$ of $T(2n+1,2m,2r)$ and two consecutive horizontal edges of $M\cap \nabla(C_i)$ such that they are on the same side. Suppose to the contrary that any two consecutive horizontal edges of $M\cap \nabla(C)$ are on different sides for an $I$-cycle $C$.
Since $l\geq 2$, there exists an $I$-cycle $C_i$ such that $M\cap \nabla(C_i)\neq \emptyset$, where $i\in Z_{(2r,2m)}$. Combining that $C_i$ is an even cycle, $M\cap \nabla(C_i)$ contains an even number of  edges. Since any two consecutive horizontal edges are on different sides, $M$ contains horizontal edges of $R_{i,i+1}$ and also those of $R_{i-1,i}$ and such two sets of horizontal edges are equal in number. Hence, each $M\cap \nabla(C_i)$ contains the same number of edges for $i\in Z_{(2r,2m)}$ and half of which belong to $R_{i-1,i}$ and the others belong to $R_{i,i+1}$.

Since an $I$-cycle contains $\frac{m(2n+1)}{(r,m)}$ vertices, $\frac{m}{(r,m)}$ is an even number. Label the vertices of $C_x$ as $1^x,2^x,\dots,(\frac{m(2n+1)}{(r,m)})^x$ where $1^x=v_{0,x}$, $2^x=v_{1,x}$ and so on. Then Lemma \ref{prop} holds. We claim that $x$ and $y$ have different parity for a horizontal edge $x^{(2r,2m)-1}y^{0}$.
Since $(2r,2m)=2rk$ (mod $2m$), we have $rk-(r,m)= mp$ for some non-negative integer $p$. Thus $\frac{r}{(r,m)}k= \frac{m}{(r,m)}p+1$. Since $\frac{m}{(r,m)}$ is even, $\frac{r}{(r,m)}\cdot k$ is odd and $k$ is an odd number. Hence $\frac{m}{(r,m)}-k$ is also an odd number. By (\romannumeral2) in Lemma \ref{prop}, $x$ and $y$ have different parity.

Without loss of generality, assume that all edges of $M\cap E(R_{0,1})$ have the form $x^0x^1$ and $x$ is an odd number. Then all edges of $M\cap E(R_{1,2})$ have the form $y^1y^2$ where $y$ is an even number. Successively,
all edges of $M\cap E(R_{(2r,2m)-1,0})$ have the form $l^{(2r,2m)-1}p^0$ where $l$ is an even number and $p$ is an odd number by the claim. Now end vertices of all edges of $M\cap \nabla(C_0)$ on $C_0$ are represented as $x^0$ and all $x$ are odd. That is to say, the other vertices of $C_0$ are all incident to vertical edges of $M$, which is a contradiction.

By Cases 1 and 2, there exists two consecutive horizontal edges of $M\cap \nabla(C_i)$ on the same side. By Lemma \ref{l3}, we obtain a perfect matching $M'$ by some flips from $M$ such that $l'=l-2$. By the induction hypothesis, we obtain the perfect matching $M_1$ or $M_2$ from $M'$ by some flips. Thus, by some flips from $M$, we can obtain  $M_1$ or $M_2$.
\end{proof}

\begin{thm}\label{t21}For $n\geq1$, $m\geq 2$ and $1\leq r\leq m$, $R_t(T(2n+1,2m,2r))$ consists of two isomorphic components.
\end{thm}
\begin{proof}By Lemmas \ref{lem11} and \ref{l1}, $R_t(T(2n+1,2m,2r))$ has exactly two components containing $M_1$ and $M_2$ separately, which are denoted by $H_1$ and $H_2$.
By Lemma \ref{auto}, $\varphi$ is an automorphism on $T(2n+1,2m,2r)$ which preserves the faces. Since $\varphi(M_1)=M_2$, by a similar proof as that of
Theorem 4.4 in \cite{LW}, we know that $\varphi$ preserves the flip operation. Thus, $\varphi$ induces an isomorphism from $H_1$ to $H_2$.
\end{proof}

Now we have proved Theorem \ref{comnum} for $T(2n+1,2m,2r)$. In next subsection, we will prove Theorem \ref{comnum} for other non-bipartite quadriculated tori $T(2m,2n+1,r)$ and $T(2n,2m,2r-1)$.

\subsection{\small Proof for quadriculated tori $T(2m,2n+1,r)$ and $T(2n,2m,2r-1)$}
By Proposition \ref{drawing}, we will prove that $T(2m,2n+1,r)$ and $T(2n,2m,2r-1)$ can be reduced into some $T(2n'+1,2m',2r')$. By the result of $T(2n'+1,2m',2r')$, we can prove Theorem \ref{comnum}.

\begin{lem}\label{huafa} Let $n\geq 1$, $m\geq 2$, $1\leq r\leq 2n+1$ and $1\leq s\leq m$ be integers. Then $T(2m,2n+1,r)$ and $T(2n,2m,2s-1)$ can be reduced into some $T(n',m',r')$, where $n'\geq 1$ is odd, $m'\geq 2$ and $2 \leq r'\leq m'$ are even.
\end{lem}
\begin{proof}By Proposition \ref{drawing}, $T(2m,2n+1,r)$ can be reduced into $T(n', m',r')$, where $$n'=(r,2n+1), m'=\frac{2m(2n+1)}{(r,2n+1)}, r'=2m(\frac{2n+1}{(r,2n+1)}-k),$$ and  $0\leq k\leq \frac{2n+1}{(r,2n+1)}-1$ is an integer so that the equation $(r,2n+1)\equiv rk$ (mod $2n+1$) holds.
Then $n'\geq 1$ is odd, $m'\geq 2m\geq 4$ and $4\leq r'\leq m'$ are even.

By Proposition \ref{drawing}, $T(2n,2m,2s-1)$ can be reduced into $T(n',m',r')$, where $$n'=(2s-1,2m), m'=\frac{4mn}{(2s-1,2m)}, r'=2n(\frac{2m}{(2s-1,2m)}-k),$$ and $0\leq k\leq \frac{2m}{(2s-1,2m)}-1$ satisfies the equation $(2s-1,2m)\equiv (2s-1)k$ (mod $2m$). Then $n'\geq 1$ is odd, $m'\geq 2n\geq 2$ and $2\leq r'\leq m'$ are even.
\end{proof}

By Lemma \ref{huafa}, $T(2m,2n+1,r)$ and $T(2n,2m,2s-1)$ may be reduced into $T(1,2m',2r')$ which may be non-simple graph (containing a loop or multiple edge).
By a simple argument, we obtain that $T(1,2m,t)$ is simple if and only if $t$ is not equal to 1,
$2m-1$, $m$ and $2m$. Since $T(1,2m,2r)$ thins the quadriculated torus $T(3,2m,2r)$ and keeps its original structure, Theorem \ref{t21} holds for $r\neq m$ if $m$ is odd,
and $r\notin \{\frac{m}{2},m\}$ otherwise.
\begin{remark}\label{re12mr} Theorem \ref{t21} holds for simple quadriculated torus $T(1,2m,2r)$, that is $r\neq m$ if $m$ is odd, and $r\notin \{\frac{m}{2},m\}$ otherwise.
\end{remark}

Now we prove Theorem \ref{comnum} for quadriculated tori $T(2m,2n+1,r)$ and $T(2n,2m,2s-1)$. By Lemma \ref{huafa}, $T(2m,2n+1,r)$ and $T(2n,2m,2s-1)$ can be reduced into some  $T(n',m',r')$, where $n'\geq 1$ is odd, $m'\geq 2$ and $2 \leq r'\leq m'$ are even.
If $n'\geq 3$, then the flip graph of $T(n',m',r')$ consists of two isomorphic components by Theorem \ref{t21}.
Otherwise, $n'=1$ and $T(1,2m',2r')$ may contain loops or multiple edges.
Since $T(2m,2n+1,r)$ and $T(2n,2m,2s-1)$ is simple, so is $T(1,2m',2r')$. Combining Remark \ref{re12mr},
we obtain that Theorem \ref{comnum} holds.~~~~~~~~~~~~~~~~~~~~~~~~~~$\square$

\section{\normalsize Bipartite quadriculated tori}
In the last section, we have proved that flip graphs of non-bipartite quadriculated tori consist of two isomorphic components. Whereas, flip graphs of bipartite quadriculated tori have more complicated cases.

A \emph{ladder} \cite{T1} is a sequence of parallel dominoes side by side such that two neighboring dominoes always touch along one edge of the longer side, each domino in the ladder has two neighbors in it and theses two neighbors touch the domino at different squares.
See the first 8 tilings in Figure \ref{eg3732}, where each tiling consists of two ladders colored with yellow and blue, respectively. We can see that ladders are totally immune to flips.
In fact, Saldanha, Tomei, Casarin and Romualdo \cite{T1} used (co)homology to give a necessary and sufficient condition for two tilings in the same component of the flip graph.
\begin{thm}\rm{\cite{T1}}\label{bicomponent} Two domino tilings $t_1$ and $t_2$ of a quadriculated surface are in the same component of its flip graph if and only if $[t_1-t_2]=0$ and $t_1$ and $t_2$ have precisely the same ladders.
\end{thm}

Furthermore, Saldanha et al. \cite{T1} obtained that components of flip graphs of bipartite quadriculated tori are homotopically equivalent to points or circles.

With the computer, we obtain the components of flip graphs for bipartite quadriculated tori $T(3,4,1)$, $T(4,4,2)$ and $T(4,4,4)$.

\begin{figure}[h]
\centering
\includegraphics[height=5cm,width=16.5cm]{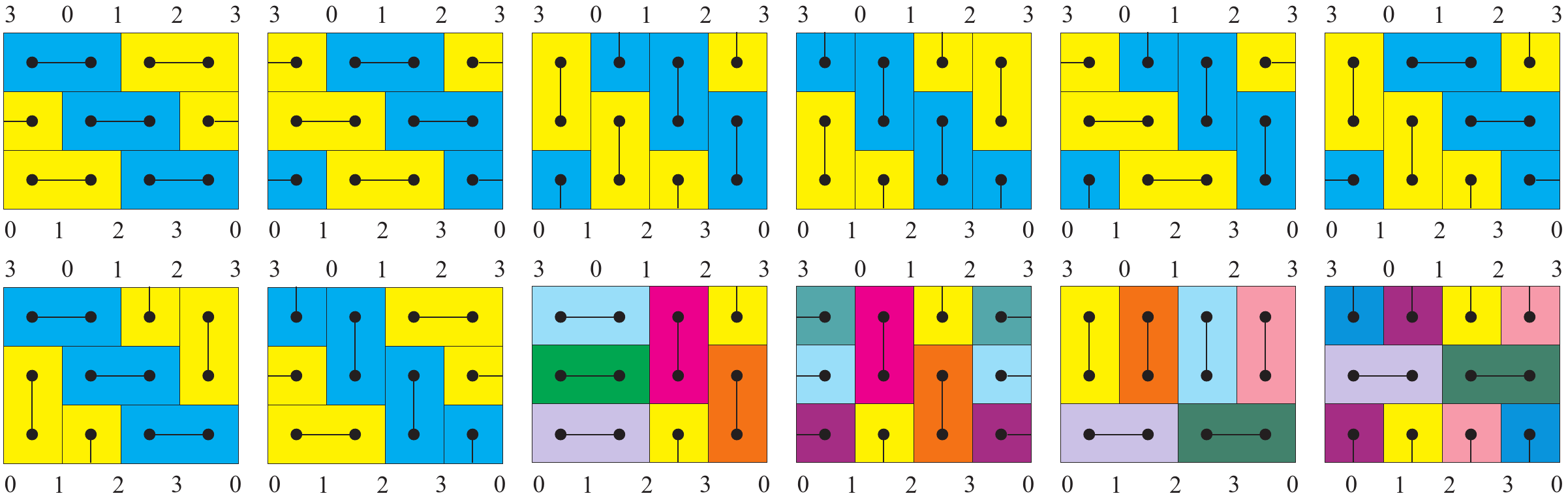}
\caption{\label{eg3732}Twelve tilings which respectively are from 12 different components of $\mathcal{T}(T(3,4,1))$.}
\end{figure}
\begin{ex} \label{ex341} $\mathcal{T}(T(3,4,1))$ contains 80 vertices, which form 8 singletons and 4 non-singleton components shown as Figure \ref{eg3732}, where the first 8 tilings themselves form singletons, and the remaining tilings are from 4 different non-singleton components.
\end{ex}

\begin{ex} \label{ex442} $\mathcal{T}(T(4,4,2))$ contains 260 vertices and 11 components shown as Figure \ref{eg4442}, where the first 4 tilings themselves form singletons, and the remaining tilings are from 7 different non-singleton components.
\begin{figure}[h]
\centering
\includegraphics[height=6.2cm,width=16cm]{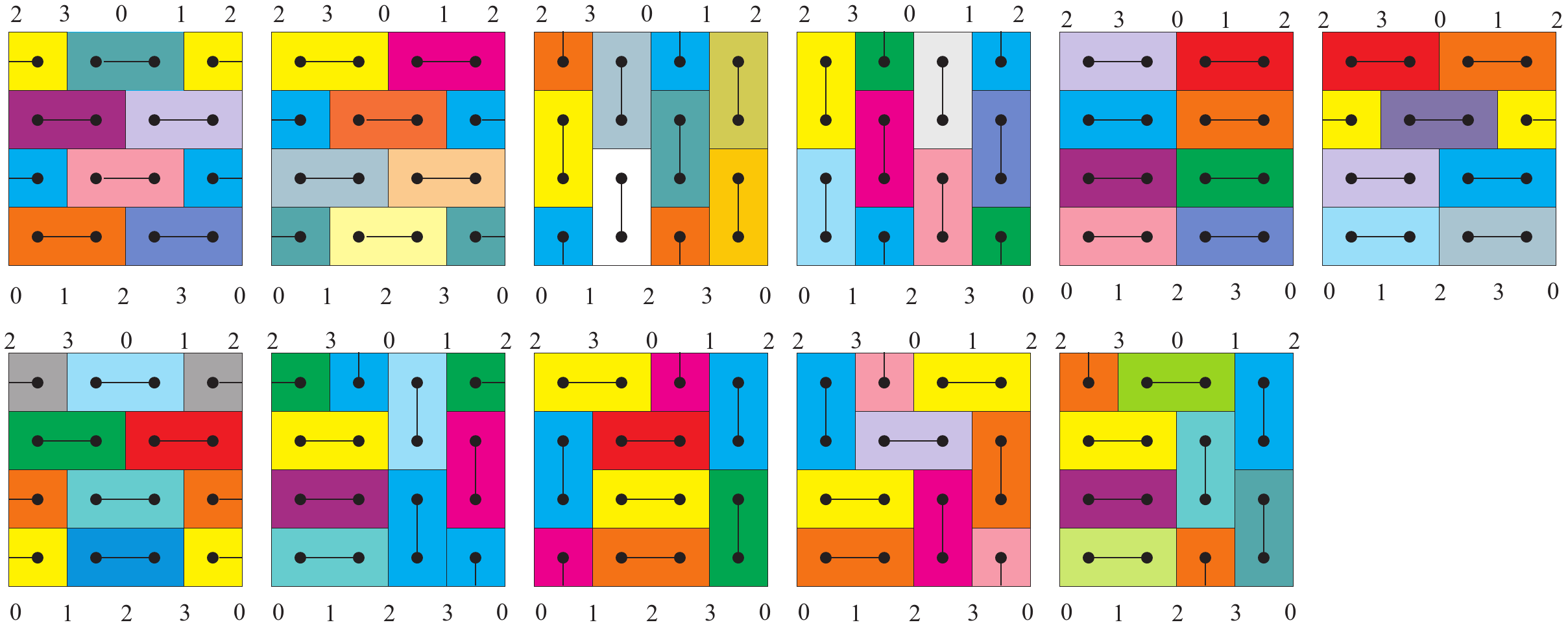}
\caption{\label{eg4442}Eleven tilings which respectively are from 11 different components  of $\mathcal{T}(T(4,4,2))$.}
\end{figure}
\end{ex}

\begin{ex} \label{ex444} $\mathcal{T}(T(4,4,4))$ contains 272 vertices and 17 components shown as Figure \ref{eg4442}, where the first 12 tilings themselves form singletons, and the remaining tilings are from 5 different non-singleton components.
\begin{figure}[h]
\centering
\includegraphics[height=7.3cm,width=15cm]{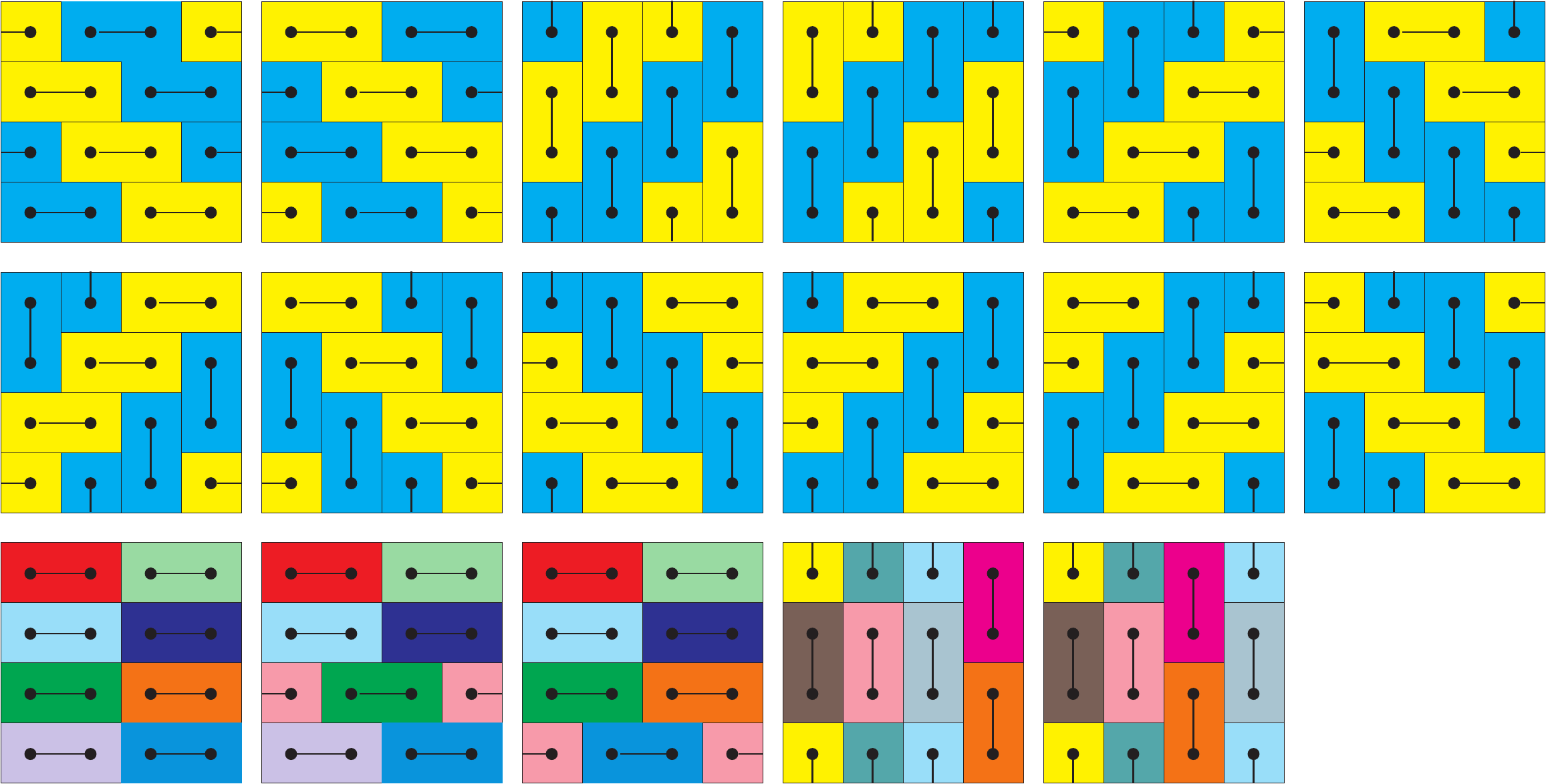}
\caption{\label{eg4444} Seventeen tilings which respectively are from 17 different components  of $\mathcal{T}(T(4,4,4))$. }
\end{figure}
\end{ex}

By the above three examples, we find that their flip graphs contain at least 4 singletons and 4 non-singleton components. In the sequel, we will discuss singletons and non-singleton components, respectively. For convenience, we define some sets of edges for $T(n,m,r)$. For $i\in Z_n$, let $X_{i}=\{v_{i,2k}v_{i,2k+1}|k\in Z_{\frac{m}{2}}\}$ and $Y_{i}=\{v_{i,2k+1}v_{i,2k+2}|k\in Z_{\frac{m}{2}}\}$. For $j\in Z_m$, let $F_{j}=\{v_{2k+1,j}v_{2k+2,j}|k\in Z_{\lceil\frac{n-2}{2}\rceil-1}\}$,  $F'_{j}=F_j\cup \{v_{n-1,m-r+j}v_{0,j}\}$ and $W_{j}=\{v_{2k,j}v_{2k+1,j}|k\in Z_{\lfloor\frac{n}{2}\rfloor-1}\}$.
\begin{thm}\label{singleton} Flip graphs of bipartite quadriculated tori contain at least four singletons.
\end{thm}
\begin{proof} Let $T(n,m,r)$ be a bipartite quadriculated torus, where $n,m\geq 3$ and $1\leq  r\leq m$. By Theorem \ref{bipartite}, both of $m$ and $n+r$ are even.
Let $t^{1}=X_0\cup Y_1\cup X_2\cup Y_3\cup \cdots \cup \sigma_{n-1}$, where if $n$ is odd then $\sigma=X$, otherwise, $\sigma=Y$.
Let $t^{2}=Y_0\cup X_1\cup Y_2\cup X_3\cup \cdots \cup \sigma_{n-1}$, where if $n$ is odd then $\sigma=Y$, otherwise, $\sigma=X$. Let  $t^{3}=F'_0\cup W_1\cup F'_2\cup W_3\cup \cdots \cup F'_{m-2}\cup W_{m-1}$ and $t^{4}=W_0\cup F'_1\cup W_2\cup F'_3\cup \cdots \cup W_{m-2}\cup F'_{m-1}$. Obviously, there is no flips can be done on above four flips, and they form singletons themselves of the flip graph.
\end{proof}
By Example \ref{ex442}, the lower bound in Theorem \ref{singleton} is tight.
For $T(2n,2n,2n)$, we can obtain a better lower bound for the number of singletons of the flip graph.
\begin{prop}\label{non-singleton3} For $n\geq 2$, the flip graph of $T(2n,2n,2n)$ have at least $4+2^{n+1}$ singletons.
\end{prop}
\begin{proof}By Example \ref{ex444}, the flip graph of $T(4,4,4)$ has 12 singletons which correspond to tilings constructed from two ladders (colored with yellow and blue in Figure \ref{eg4444}, respectively). The first 4 tilings in Figure \ref{eg4444} correspond to the four singletons defined in the proof of Theorem \ref{singleton}. In the following, we will prove that starting with the fifth singleton in Figure \ref{eg4444}, we will produce two singletons of $T(6,6,6)$.
\begin{figure}[h]
\centering
\includegraphics[height=4.2cm,width=16cm]{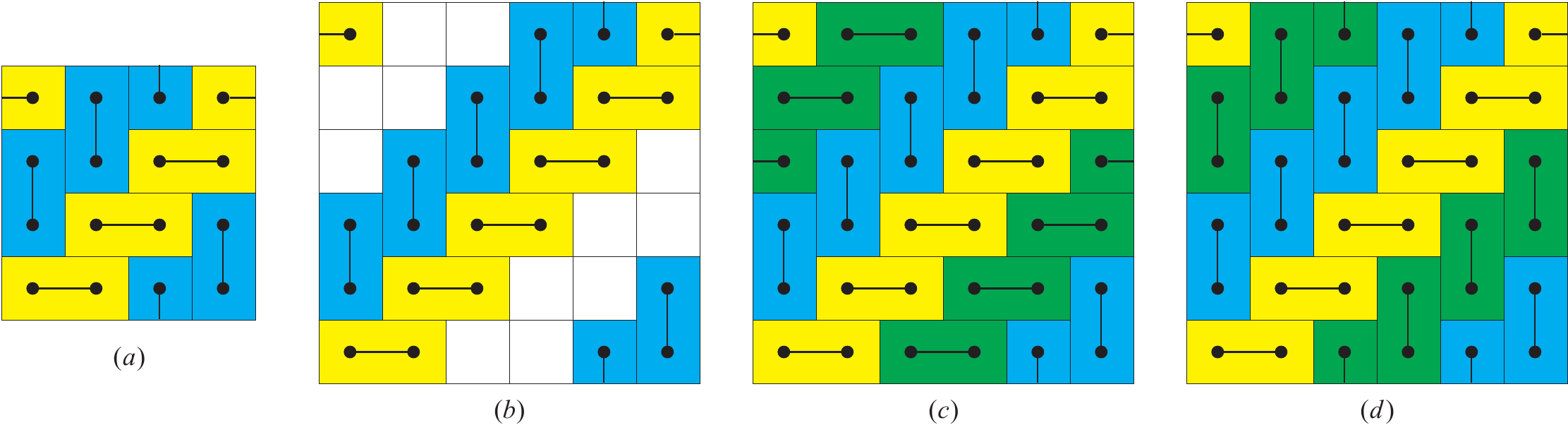}
\caption{\label{eg666} Generalization of a singleton of $T(4,4,4)$ to two singletons of $T(6,6,6)$.}
\end{figure}
First we generalize the two ladders consisting of the singleton (shown in Figure \ref{eg666}(a)) to two ladders of $T(6,6,6)$ (shown in Figure \ref{eg666}(b)), then there remain a ladder of $T(6,6,6)$ which can be colored with green by two methods (shown as (c) and (d) in Figure \ref{eg666}).

For each singleton shown in Figure \ref{eg4444} except for the first 4 ones, we can get two singletons of $T(6,6,6)$. Thus, we can obtain $8\times 2$ singletons of $T(6,6,6)$, and $8\times 2^2$ singletons of $T(8,8,8)$. Proceeding like this, we will obtain $8\times 2^{n-2}$ singletons of $T(2n,2n,2n)$. Adding the 4 singletons $t^1$, $t^2$, $t^3$ and $t^4$ in the proof of Theorem \ref{singleton}, we can obtain at least $4+2^{n+1}$ singletons.
\end{proof}

To consider non-singleton components of the flip graph of bipartite quadriculated tori $T(n,m,r)$ where $m$ and $n+r$ are even, we associate $T(n,m,r)$  a \emph{CW-complex} $\boldsymbol{T}$ \cite{T1}: The 0-cells are vertices of $T(n,m,r)$, the 1-cells are edges between vertices, and the 2-cells are unit squares. We know that the 1-dimensional homological group $H_1(\boldsymbol{T},Z)=Z^2$.

Given a 2-coloring of $T(n,m,r)$, each edge (or domino) as a 1-cell in $\boldsymbol{T}$ is oriented from black to white, and a perfect matching (or domino tiling) corresponds to a 1-chain in $C_1(\boldsymbol{T},Z)$. For domino tilings $t_1$ and $t_2$, the difference $t_1-t_2$ consists of disjoint directed cycles in $T(n,m,r)$, and is a closed 1-chain in $C_1(\boldsymbol{T},Z)$ (i.e. $\partial_1(t_1-t_2)=0$). The corresponding homology class $[t_1-t_2]$ is in $H_1(\boldsymbol{T},Z)$. For example,  we give two tilings $t_1$, $t_2$ in Figure \ref{homologyclass} (a) and (b), and the corresponding homology class $[t_1-t_2]$ in Figure \ref{homologyclass}(c).
\begin{figure}[h]
\centering
\includegraphics[height=2.8cm,width=10cm]{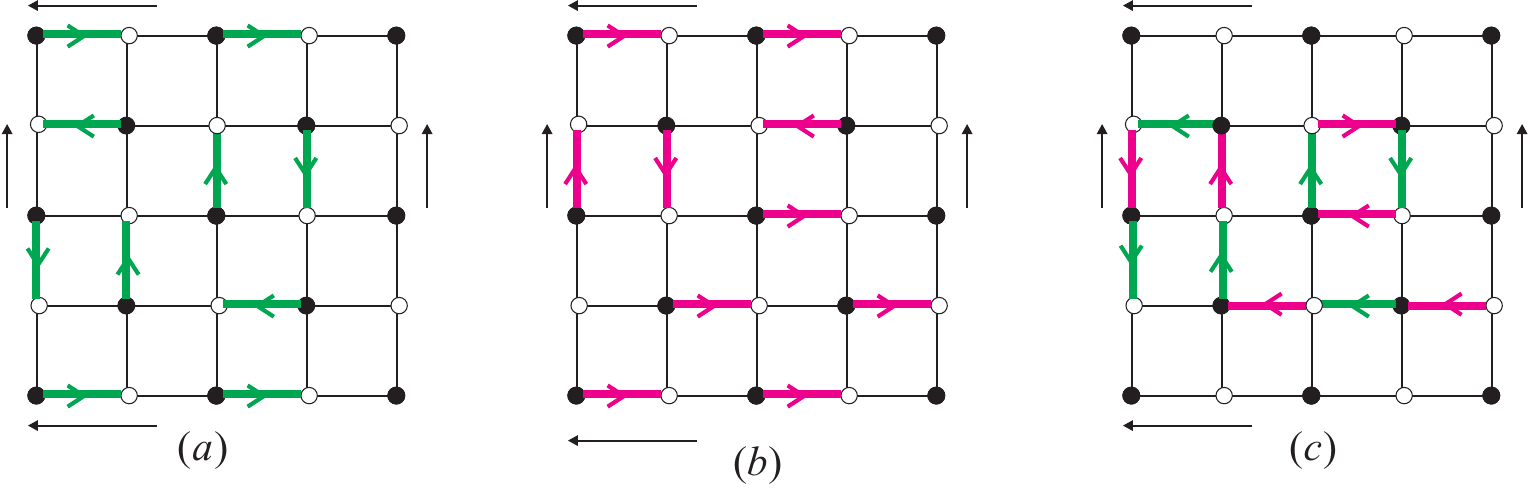}
\caption{\label{homologyclass} Two tilings $t_1$, $t_2$ and the corresponding homology class $[t_1-t_2]$.}
\end{figure}

Freire et al. \cite{LW17} defined and developed the concept of flux via homology theory.
Given a fixed tiling $t_\oplus$ as a base tiling, define \emph{flux} of a tiling $t$ as $$\text{Flux}(t)=[t-t_\oplus] \in H_1(\boldsymbol{T},Z).$$
By Theorem \ref{bicomponent}, we obtain a sufficient condition for two tilings in different components.
\begin{cor} \label{flipconnected}For tilings $t_1$ and $t_2$ of $T$,  if $\text{Flux}(t_1)\neq \text{Flux}(t_2)$, then $t_1$ and $t_2$ lie in different components of the flip graph of $T(n,m,r)$.
\end{cor}
\begin{proof} By properties of the homology class, we have $$[t_1-t_2]=[t_1-t_\oplus]+[t_\oplus-t_2]=[t_1-t_\oplus]-[t_2-t_\oplus]=\text{Flux}(t_1)- \text{Flux}(t_2)\neq 0,$$ which implies that $t_1$ and $t_2$ lie in different components of the flip graph by Theorem \ref{bicomponent}.
\end{proof}

This result together with some special tilings gives an approach to amount to the components of flip graphs for bipartite quadriculated tori.

For a bipartite quadriculated tori $T(n,m,r)$ where $m$ and $n+r$ are even, we define two closed 1-chains: $$z_0=v_{n-1,0}v_{n-1,1}+v_{n-1,1}v_{n-1,2}+\cdots +v_{n-1,m-1}v_{n-1,0},$$ $$z'_0=v_{n-1,0}v_{n-2,0}+v_{n-2,0}v_{n-3,0}+\cdots +v_{0,0}v_{n-1,m-r}+v_{n-1,m-r}v_{n-1,m-r+1}+\cdots +v_{n-1,m-1}v_{n-1,0}.$$ Then $H_1(T,Z)$ is a free abelian group based on $[z_0]$ and $[z'_0]$, which is represented as $[z_0]=(1,0)$ and $[z'_0]=(0,1)$. For $i\in Z_{n}$, let $$z_i=v_{i,0}v_{i,1}+v_{i,1}v_{i,2}+\cdots+ v_{i,m-1}v_{i,0}.$$ By a simple argument, we obtain that $[z_i]-[z_0]=0$.

For example, the last five tilings in Figure \ref{eg4444} in turn are denoted by $t_1=t_\oplus$, $t_2$, \dots, $t_5$ respectively. Then $\text{Flux}(t_1)=(0,0)$, $\text{Flux}(t_2)=(1,0)$, $\text{Flux}(t_3)=(-1,0)$, $\text{Flux}(t_4)=(0,-1)$ and $\text{Flux}(t_5)=(0,1)$. By Corollary \ref{flipconnected}, such five tilings belong to different components of $\mathcal{T}(T(4,4,4))$.

Let $t_\oplus=E_0\cup E_2\cup \dots \cup E_{m-2}$ be a base tiling of $T(n,m,r)$. Then we can obtain the following result.
\begin{thm}\label{non-singleton} For $n\geq 3$, $m\geq 4$ and $1\leq r\leq m$, let $T(n,m,r)$ be a bipartite quadriculated torus where $m$ and $n+r$ are even. Then $\mathcal{T}(T(n,m,r))$ has at least $2\lfloor\frac{n}{2}\rfloor-1$ non-singleton components.
\end{thm}
\begin{proof}Since $t_\oplus$ is a tiling of $T(n,m,r)$ adjacent to at least 4 other tilings, we obtain that $\mathcal{T}(T(n,m,r))$ has at least one non-singleton component. Next we assume that $n\geq 4$.
For $k\in Z_{\lfloor\frac{n}{2}\rfloor-1}$, let $$t_k=(\bigcup_{j\in Z_{k+1}}Y_{2j+1})\,\bigcup\,(\bigcup_{i\in I}X_{i}) \text{ where } I=Z_{n}\setminus \{2j+1|j\in Z_{k+1}\}, \text{ and}$$
$$t'_k=(\bigcup_{j\in Z_{k+1}}Y_{2j})\,\bigcup\,(\bigcup_{i\in I}X_{i}) \text{ where } I=Z_{n}\setminus \{2j|j\in Z_{k+1}\}$$ be $2\lfloor\frac{n}{2}\rfloor-2$ tilings of $T(n,m,r)$ consisting of horizontal edges. Obviously, $t'_k$ is obtained by translating $t_k$ one square up for $k\in Z_{\lfloor\frac{n}{2}\rfloor-1}$. For example, for $T(7,4,1)$, tilings $t_0$ and $t_1$ are shown as (a) and (b), and $t'_0$ and $t'_1$ are shown as (c) and (d) in Figure \ref{obsess}.
\begin{figure}[h]
\centering
\includegraphics[height=4cm,width=12cm]{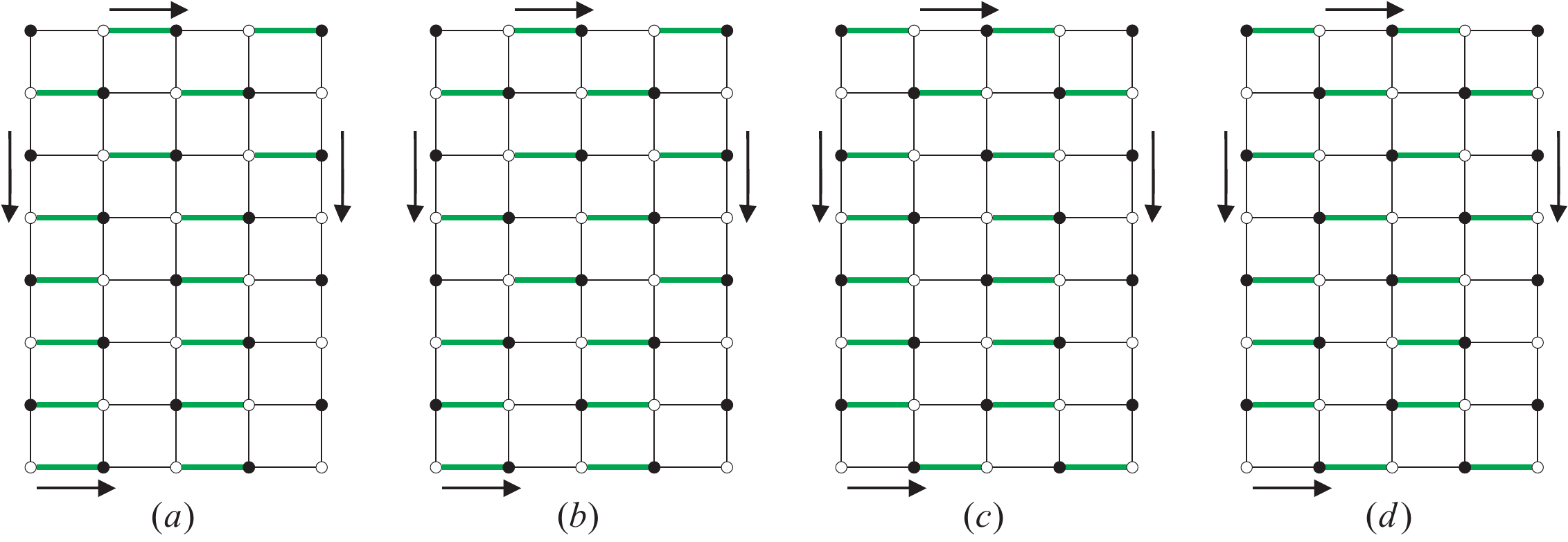}
\caption{\label{obsess} Four tilings $t_0$, $t_1$, $t'_0$ and $t'_1$ of $T(7,4,1)$.}
\end{figure}

By the induction on $k$, we are to prove that $\text{Flux}(t_k)=(-k-1,0)$ and $\text{Flux}(t'_k)=(k+1,0)$ for $k\in Z_{\lfloor\frac{n}{2}\rfloor-1}$.
For $k=0$, $\text{Flux}(t_0)=[t_0-t_\oplus]$ is a closed 1-chain $$v_{1,0}v_{1,m-1}+v_{1,m-1}v_{1,m-2}+\cdots +v_{1,1}v_{1,0}.$$ That is, $\text{Flux}(t_0)=[t_0-t_\oplus]=(-1,0)$. Suppose that $\text{Flux}(t_s)=(-s-1,0)$ for $0\leq s<\lfloor\frac{n}{2}\rfloor-1$. We will prove that $\text{Flux}(t_{s+1})=(-s-2,0)$. Noting that all edges of $t_s$ and $t_{s+1}$ are the same except for horizontal edges lying in the $2s+5$ rows of $T(n,m,r)$ form top to bottom, we obtain that $[t_{s+1}-t_s]$ is a closed 1-chain $$v_{2s+3,0}v_{2s+3,m-1}+v_{2s+3,m-1}v_{2s+3,m-2}+\cdots +v_{2s+3,1}v_{2s+3,0}.$$ Thus, $[t_{s+1}-t_s]=(-1,0)$, and $$\text{Flux}(t_{s+1})=[t_{s+1}-t_\oplus]=[t_{s+1}-t_s]+[t_{s}-t_\oplus]=(-1,0)+\text{Flux}(t_{s})=(-s-2,0).$$
By the induction, we obtain that $\text{Flux}(t_k)=(-k-1,0)$ for $k\in Z_{\lfloor\frac{n}{2}\rfloor-1}$.

Obviously, $t'_k$ is obtained by translating $t_k$ one square up. Hence $[t'_0-t_\oplus]$ and $[t'_{s+1}-t'_s]$ are also 1-chains consisting of all horizontal edges of $T(n,m,r)$ on the second and $2s+4$ rows, respectively, from top to bottom whose direction is horizontally to right. Therefore, we have $[t'_0-t_\oplus]=(1,0)$ and $[t'_{s+1}-t'_s]=(1,0)$. By the induction on $k$, we can prove that $\text{Flux}(t'_k)=(k+1,0)$ for $k\in Z_{\lfloor\frac{n}{2}\rfloor-1}$.

Note that $\text{Flux}(t_\oplus)=(0,0)$. Since values of flux of such $2\lfloor\frac{n}{2}\rfloor-1$ tilings are different pairwise, the flip graph of $T(n,m,r)$ has at least $2\lfloor\frac{n}{2}\rfloor-1$ non-singleton components by Corollary \ref{flipconnected}.
\end{proof}

For $T(2n,2m,2m)$, we can obtain a better lower bound for the number of non-singleton components of the flip graph.

\begin{thm}\label{non-singleton3} For $n, m\geq 2$, $\mathcal{T}(T(2n,2m,2m))$ has at least $2n+2m-3$ non-singleton components.
\end{thm}
\begin{proof}
For $k\in Z_{m-1}$, let $$t^k=(\bigcup_{j\in Z_{k+1}}F'_{2j+1})\,\bigcup\,(\bigcup_{i\in I}W_{i}) \text{ where }I=Z_{2m}\setminus \{2j+1|j\in Z_{k+1}\}, \text{ and }$$ $$t'^k=(\bigcup_{j\in Z_{k+1}}F'_{2j})\,\bigcup\,(\bigcup_{i\in I}W_{i}) \text{ where }I=Z_{2m}\setminus \{2j|j\in Z_{k+1}\}$$ be $2m-2$ tilings of $T(2n,2m,2m)$ consisting of vertical edges.
\begin{figure}[h]
\centering
\includegraphics[height=3cm,width=16.5cm]{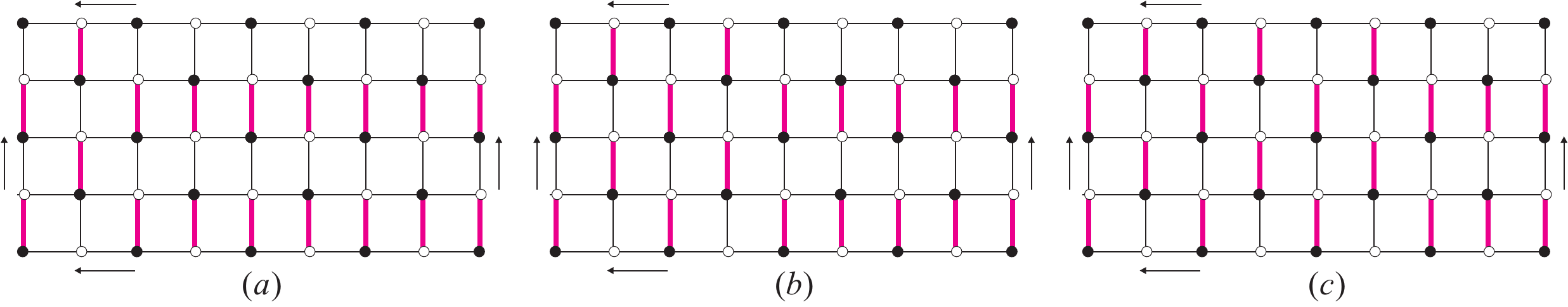}
\caption{\label{homoeven}Three tilings $t^0$, $t^1$ and $t^2$ of $T(4,8,8)$.}
\end{figure}
Obviously, $t'^k$ is obtained by translating $t^k$ one square left for $k\in Z_{m-1}$. For example, tilings $t^0$, $t^1$ and $t^2$ of $T(4,8,8)$ are shown as (a), (b) and (c) in Figure \ref{obsess}.
Next we are to prove that $\text{Flux}(t^k)=(0,k+1)$ and $\text{Flux}(t'^k)=(0,-k-1)$ for $k\in Z_{m-1}$ by the induction on $k$.

For $k=0$, we consider the homology class $[t^0-t_\oplus]$ shown as in Figure \ref{homoeven}(a), which consists of boundaries of $n(m-1)$ directed unit squares and a closed 1-chain with $4n$ directed edges.
\begin{figure}[h]
\centering
\includegraphics[height=3cm,width=12cm]{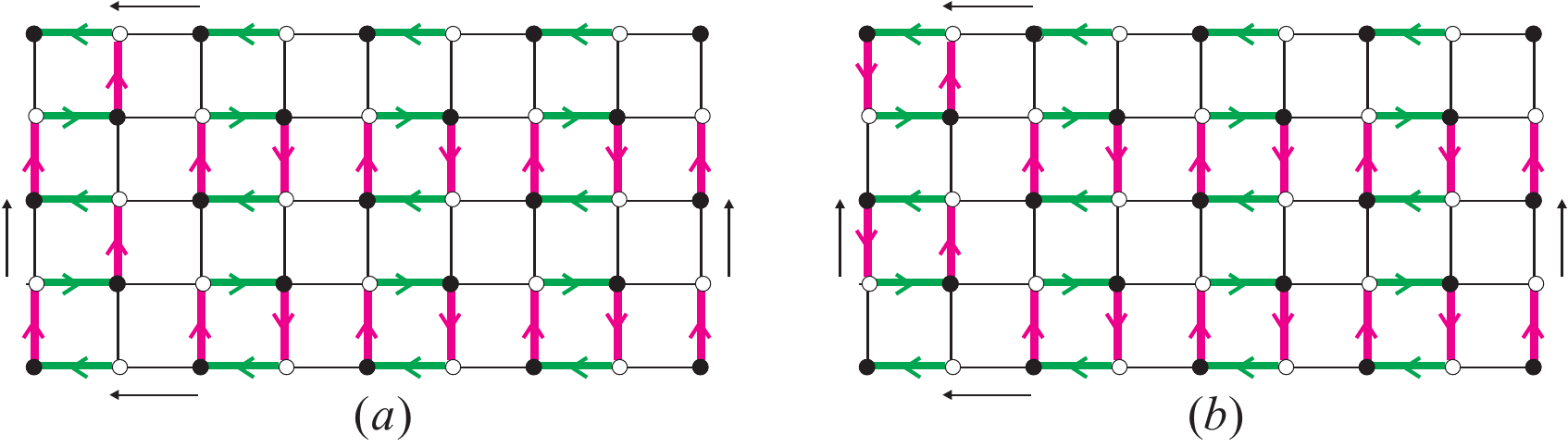}
\caption{\label{homoeven}Closed 1-chain $[t^0-t_\oplus]$ and $[t^0-t_\oplus]-[z'_0]$.}
\end{figure}

Recall that $$z'_0=v_{2n-1,0}v_{2n-2,0}+v_{2n-2,0}v_{2n-3,0}+\cdots +v_{1,0}v_{0,0}$$ is a 1-chain with $[z'_0]=(0,1)$. Since $[t^0-t_\oplus]-[z'_0]$ consists of boundaries of $nm$ directed unit squares shown in Figure \ref{homoeven}(b), which belongs to $B_1(T,Z)$, we have $$\text{Flux}(t^0)=[t^0-t_\oplus]=[z'_0]=(0,1).$$

Suppose that $\text{Flux}(t^s)=(0, s+1)$ for $0\leq s<m-1$. We will prove that $\text{Flux}(t^{s+1})=(0, s+2)$. Noting that all edges of $t^s$ and $t^{s+1}$ are the same except for vertical edges lying in $(2s+3)$-column of $T(2n,2m,2m)$. Hence $[t^{s+1}-t^s]$ is a closed 1-chain $$v_{2n-1,2s+3}v_{2n-2,2s+3}+v_{2n-2,2s+3}v_{2n-3,2s+3}+\cdots + v_{0,2s+3}v_{2n-1,2s+3},$$ and $[t^{s+1}-t^{s}]=(0,1)$. Thus, we have $$\text{Flux}(t^{s+1})=[t^{s+1}-t_\oplus]=[t^{s+1}-t^s]+[t^{s}-t_\oplus]=(0,1)+\text{Flux}(t^{s})=(0,s+2).$$

Obviously, $t'^k$ is obtained by translating $t^k$ one square left. Hence $[t'^0-t_\oplus]=[-z'_0]$, and $\text{Flux}(t'^0)=[t'^0-t_\oplus]=-[z'_0]=(0,-1)$.
Since $[t'^{s+1}-t'^s]$ is a closed 1-chain $$v_{2n-1,2s+2}v_{0,2s+2}+v_{0,2s+2}v_{1,2s+2}+\cdots +v_{2n-2,2s+2}v_{2n-1,2s+2},$$ and $[t'^{s+1}-t'^s]=[-z'_0]=(0,-1)$.
Therefore, we have $$\text{Flux}(t'^{s+1})=[t'^{s+1}-t_\oplus]=[t'^{s+1}-t'^s]+[t'^{s}-t_\oplus]=(0,-1)+\text{Flux}(t'^{s})=(0,-s-2).$$
By the induction on $k$, we prove the result.

By the proof of Theorem \ref{non-singleton}, we obtain that $\text{Flux}(t_k)=(k+1,0)$ and $\text{Flux}(t'_k)=(-k-1,0)$ for $k\in Z_{n-1}$ and $\text{Flux}(t_\oplus)=(0,0)$.
Since values of flux of such $2n+2m-3$ tilings are different pairwise, the flip graph of $T(2n,2m,2m)$ has at least $2n+2m-3$ non-singleton components by Corollary \ref{flipconnected}.
\end{proof}

Example \ref{ex444} shows that the lower bound in Theorem \ref{non-singleton3} is tight.

\section{\normalsize An application to forcing spectra of quadriculated tori}
We now prove Theorem \ref{qpp1} by using a lemma due to Afshani et al. \cite{5}.
\begin{lem}\rm{\cite{5}} \label{4.3} If $M$ is a perfect matching of a graph $G$ and $C$ is an $M$-alternating cycle of length 4, then the symmetric difference on $M$ along $C$ does not change the forcing number by more than 1.
\end{lem}

\noindent\textbf{Proof of  Theorem \ref{qpp1}.}

By Theorem \ref{bipartite}, it suffices to prove that Theorem \ref{qpp1} holds separately for quadriculated torus $T(2n+1,2m,2r)$,
$T(2n,2m,2r-1)$ and $T(2m,2n+1,r)$.

By Theorem \ref{t21} and Remark \ref{re12mr}, simple quadriculated torus $R_t(T(2n+1,2m,2r))$ has exactly two components $H_1$ and $H_2$ where $n\geq 0$, $m\geq2$ and $1\leq r\leq m$. By the definition of resonance graph, any two perfect matchings of $H_i$ can
be obtained from each other by some flips along some faces of $T(2n+1,2m,2r)$ for $i\in\{1,2\}$. By Lemma \ref{4.3}, the forcing numbers of all perfect matchings in $H_i$
form an integer-interval.

Two perfect matchings $M^1$ and $M^2$ are called \emph{equivalent} if there is an automorphism of the graph which maps $M^1$ to $M^2$. By Lemma \ref{auto},
$\varphi$ is an automorphism of $T(2n+1,2m,2r)$. Since $\varphi(M_1)=M_2$, $M_1$ and $M_2$ are equivalent. By the proof of Theorem \ref{t21}, $\varphi$ induces an
isomorphism from $H_1$ to $H_2$. Thus, for any perfect matching $M$ of $H_1$, there exists a perfect matching $\varphi(M)$ of $H_2$ such that $M$ and $\varphi(M)$ are equivalent.
Since two equivalent perfect matchings have the same forcing number, the sets of forcing numbers of all perfect matchings in $H_1$ and $H_2$ are the same. Thus forcing numbers
of all perfect matchings of $T(2n+1,2m,2r)$ form an integer-interval. That is, the forcing spectrum of $T(2n+1,2m,2r)$ is continuous.

By Lemma \ref{huafa}, $T(2m,2n+1,r)$ and $T(2n,2m,2r-1)$ can be reduced into some quadriculated torus $T(2n'+1,2m',2r')$ where $n\geq 0$, $m\geq2$ and $1\leq r\leq m$. By above arguments, we obtain that Theorem \ref{qpp1} holds for all non-bipartite quadriculated tori. ~~~~~~~~~~~~~~~~~~~~~~~~~~~~~~~~~~~~~~~~~~~~~~~~~~~~~~~~$\square$

For bipartite quadriculated tori, their forcing spectra are not necessarily continuous.
\begin{ex}By computing forcing numbers of all 18656 perfect matchings of $T(3,10,1)$ by computer, we obtain a set $\{3,5,6,7,8\}$, which has one gap 4.

By calculating forcing numbers of all 537636 perfect matchings by computer, we obtain that $\text{Spec}(T(4,10,10))=\{4,6,7,8,9,10\}$, which has one gap 5.
\end{ex}

\noindent{\normalsize \textbf{Acknowledgments}}

The authors are grateful to the reviewer for pointing out a criterion of bipartite quadriculated tori (see Theorem \ref{bipartite}) and suggesting to consider the components of flip graphs via homology theory.
This work is supported by NSFC\,(Grant No. 12271229), start-up funds of Inner Mongolia Autonomous Region (Grant No. DC2400002165) and Inner Mongolia University of Technology (Grant No. BS2024038).


\begin{thebibliography}{99}
\setlength{\itemsep}{-.8mm}
\bibitem{5}P. Afshani, H. Hatami, E. S. Mahmoodian, On the spectrum of the forced matching number of graphs, Australas. J. Combin. 30 (2004) 147-160.
\bibitem{cyclic}George P. Barker, Manifestations of Polya's counting theorem, Linear Algebra Appl. 32 (1980) 209-232.

\bibitem{cohn}H. Cohn, R. Kenyon, J. Propp, A variational principle for domino tilings, J. Amer. Math. Soc. 14(2) (2001) 297-346.



\bibitem{Fisher}M. E. Fisher, Statistical mechanics of dimers on a plane lattice, Phys. Rev. 124 (1961) 1664-1672.
\bibitem{fou}J. C. Fournier, Combinatorics of perfect matchings in plane bipartite graphs and application to tilings, Theor. Comput. Sci. 303 (2003) 333-351.

\bibitem{LW17}J. Freire, C. Klivans, P. H. Milet, N. C. Saldanha, On the connectivity of spaces of three-dimensional domino tilings, Trans. Amer. Math. Soc. 375(3) (2022) 1579-1605.
\bibitem{G}W. Gr\"{u}ndler, Signifikante Elektronenstrukturen fur benzenoide Kohlenwasserstoffe, Wiss. Z. Univ. Halle 31 (1982) 97-116.
\bibitem{hardy}G. H. Hardy, E. M. Wright, An Introduction to the Theory of Numbers, 5th ed., Clarendon Press, Oxford, 1979.


\bibitem{Ka} P. W. Kasteleyn, The statistics of dimers on a lattice \uppercase\expandafter{\romannumeral1}. The number of dimer arrangements on a quadratic lattice, Physica 27(12) (1961) 1209-1225.
\bibitem{3}K. J. Klein, M. Randi\'{c}, Innate degree of freedom of a graph, J. Comput. Chem. 8(4) (1987) 516-521.


\bibitem{classfy}Q. Li, Study on the Matching Extension of Graphs in Surfaces, Ph.D. Thesis, Lanzhou University, 2012.
\bibitem{LC}F. Lin, A. Chen, J. Lai, Dimer problem for some three dimensional lattice graphs, Physica A 443 (2016) 347-354.
\bibitem{LW}Q. Liu, J. Wang, C. Li, H. Zhang, Components of domino tilings under flips in quadriculated cylinder and torus, preprint, 2022, https://arxiv.org/abs/2211.10935.
\bibitem{lo}M. Loebl, On the dimer problem and the ising problem in finite 3-dimensional lattices, Electron. J. Combin. 9 (2002) R30.
\bibitem{ran}M. Randi\'{c}, Resonance in catacondensed benzenoid hydrocarbons, Int. J. Quantum Chem. 63(2) (1997) 585-600.

\bibitem{RZ1}T. Regge, R. Zecchina, Exact solusion of the ising model on group lattices of genus $g> 1$, J. Math. Phys. 37 (1996) 2796-2814.
\bibitem{RZ2}T. Regge, R. Zecchina, Combinatoriacl and topological approach to the 3D ising model, J. Phys. A 33 (2000) 741-761.


\bibitem{T1}N. C. Saldanha, C. Tomei, M. A. Jr. Casarin, D. Romualdo, Spaces of domino tilings, Discrete Comput. Geom. 14(2) (1995) 207-233.
\bibitem{TF}H. N. V. Temperley, M. E. Fisher, Dimer problem in statistical mechanics-an exact result, Phil. Mag. 6(8) (1961) 1061-1063.
\bibitem{Tho}C. Thomassen, Tilings of the torus and the Klein bottle and vertex-transitive graphs on a fixed surface, Trans. Amer. Math. Soc. 323(2) (1991) 605-635.

\bibitem{Th90}W. P. Thurston, Conway's tiling groups,  Amer. Math. Monthly 97(8) (1990) 757-773.





\bibitem{Z4}H. Zhang, The connectivity of Z-transformation graphs of perfect matchings of polyominoes, Discrete Math. 158(1-3) (1996) 257-272.
\bibitem{Zh06}H. Zhang, Z-transformation graphs of perfect matchings of plane bipartite graphs: a survey, MATCH Commun. Math. Comput. Chem. 56(3) (2006) 457-476.
\bibitem{Z1}F. Zhang, X. Guo, R. Chen, Z-transformation graphs of perfect matchings of hexagonal systems, Discrete Math. 72(1-3) (1988) 405-415.
\bibitem{Z3}H. Zhang, F. Zhang, Plane elementary bipartite graphs, Discrete Appl. Math. 105(1-3) (2000) 291-311.
\bibitem{Z6}H. Zhang, F. Zhang, H. Yao, Z-transformation graphs of perfect matchings of plane bipartite graphs, Discrete Math. 276(1-3) (2004) 393-404.
\end{thebibliography}
\end{document}